\documentclass{aims} 
\usepackage{amsmath}
\usepackage{paralist}
\usepackage[misc]{ifsym}
\usepackage{epsfig} 
\usepackage{epstopdf} 
\usepackage[colorlinks=true]{hyperref}
\hypersetup{urlcolor=blue, citecolor=red}
\allowdisplaybreaks

\def\currentvolume{X}
 \def\currentissue{X}
  \def\currentyear{200X}
   \def\currentmonth{XX}
    \def\ppages{X-XX}
     \def\DOI{10.3934/xx.xxxxxxx}

\newtheorem{theorem}{Theorem}[section]

\newtheorem{lemma}[theorem]{Lemma}

\theoremstyle{definition}
\newtheorem{definition}[theorem]{Definition}
\newtheorem{remark}[theorem]{Remark}

\usepackage{amssymb}
\usepackage{verbatim}
\usepackage{latexsym}

\expandafter\let\csname equation*\endcsname\relax
\expandafter\let\csname endequation*\endcsname\relax

\usepackage{amsmath,amsthm,amsfonts} 
\usepackage{eucal}
\usepackage{cases}

\usepackage{mathtext}
\usepackage[cp1251]{inputenc}
\usepackage[T2A]{fontenc}
\usepackage{graphicx}
\usepackage{amssymb}
\usepackage{amsxtra}
\usepackage{amsthm}
\usepackage{latexsym}
\usepackage{ifthen}
\usepackage{color}
\usepackage{hyperref}

\usepackage[pagewise]{lineno}

\usepackage[numbers,sort&compress]{natbib}

\def\newDelta{\mathcal{Z}}

\newcommand{\blue}[1]{\ifmmode{\textcolor{blue}{#1}}\else {\textcolor{blue}{#1}}\fi}

\newcommand{\red}[1]{\ifmmode\mathbf{\textcolor{red}{#1}}\else \textbf{\textcolor{red}{#1}}\fi}

\title[Weak solutions of a control-volume fibre coating model]
{On weak solutions of a control-volume model for liquid films flowing down a fibre} 

\author[Roman M. Taranets, Hangjie Ji and Marina Chugunova]{}

\subjclass{Primary: 35K35; Secondary: 35K55, 35K65, 76A20, 35Q35}
\keywords{Thin films; travelling waves; fourth-order parabolic partial differential equations; lubrication theory; weak solution}

\thanks{Roman Taranets is partially supported by the European Federation of Academies of Sciences and Humanities (ALLEA) through grant EFDS-FL2-08.
Hangjie Ji is supported by NSF DMS 2309774 and NC State University Faculty Research and Professional Development (FRPD) Grant.}

\thanks{$^*$Corresponding author: Hangjie Ji}

\begin{document}

\maketitle
\centerline{\scshape
Roman M. Taranets$^{{\href{mailto:taranets\_r@yahoo.com}{\textrm{\Letter}}}1}$,
Hangjie Ji$^{{\href{mailto:hangjie\_ji@ncsu.edu}{\textrm{\Letter}}}*2}$
and Marina Chugunova$^{{\href{mailto:marina.chugunova@cgu.edu}{\textrm{\Letter}}}3}$
}

\medskip

{\footnotesize
 \centerline{$^1$Institute of Applied Mathematics and Mechanics of the NASU, G.Batyuka Str. 19, 84100, Sloviansk, Ukraine}
} 

\medskip

{\footnotesize
 \centerline{$^2$Department of Mathematics, North Carolina State University, 2311 Stinson Dr, Raleigh, NC 27695, USA}
}

\medskip
{\footnotesize
 \centerline{$^3$Institute of Mathematical Sciences, Claremont Graduate University, 150 E. 10th Str., Claremont, CA 91711, USA}
}

\bigskip



\begin{abstract}
This paper presents an analytical investigation of the solutions to a control volume model for liquid films flowing down a vertical fibre.
The evolution of the free surface is governed by a coupled system of degenerate nonlinear partial differential equations, which describe the fluid film's radius and axial velocity. We demonstrate the existence of weak solutions to this coupled system by applying a priori estimates derived from energy-entropy functionals. 
Additionally, we establish the existence of traveling wave solutions for the system. 
To illustrate our analytical findings, we present numerical studies that showcase the dynamic solutions of the partial differential equations as well as the traveling wave solutions.
\end{abstract}

\maketitle

\section{Introduction}

Thin liquid films flowing down a vertical fibre have attracted many interests in the past decades due to their importance in a variety of engineering applications, including heat and mass exchangers, thermal desalination, and vapor and CO$_2$ capturing \cite{sadeghpour2019water,zeng2019highly,zeng2018thermohydraulic,zeng2017experimental,sadeghpour2021experimental}.
These liquid films are fundamentally driven by Rayleigh-Plateau instability and gravity modulation, spontaneously exhibiting complex interfacial instability and pattern formation \cite{Qu_r__1990,quere1999fluid,sedighi2021capillary,ruyer2012wavy}.

Previous experimental works have found that the downstream flow dynamics and pattern formation highly depend on the flow rate, fibre radius, liquid properties, and inlet geometries. Specifically, three typical flow regimes have been extensively studied \cite{kliakhandler2001viscous,craster2006viscous,ruyer2008modelling,frenkel1992nonlinear,yu2013velocity}. At high flow rates, the convective instability dominates the system and irregular droplet coalescence occur frequently. For lower flow rates, the Rayleigh-Plateau regime emerges where stable travelling droplets move at a constant speed. If flow rates are further reduced, the isolated droplet regime occurs where widely-spaced droplets coexist with small amplitude wave patterns. Similar regime transitions can also be triggered by varying the nozzle diameters or imposing a gradient to the liquid property along the fibre \cite{Ji2020Modeling,ji2020thermal,ding2017three}.
A good understanding of these dynamics is critical to the design and control of engineering systems that involve a stable train of travelling droplets.

In the low Reynolds number limit, classical lubrication models have been developed for the dynamics of falling viscous liquid films along an axisymmetric cylindrical fibre. Under the long-wave approximation, the resultant evolution equations are a family of fourth-order degenerate parabolic PDEs for the fluid film thickness \cite{frenkel1992nonlinear,chang1999mechanism,kalliadasis1994drop,ji2019dynamics,marzuola2019}. These models incorporate the effects of gravity and surface tension, where the surface tension plays both stabilizing and destabilizing roles depending on the axial and azimuthal curvature of the free surface.
Numerical and analytical investigations of these models have also revealed the dependence of the droplet dynamics on the substrate effects and external physical fields \cite{ji2020travelling,ji2020thermal,quere1989making}.

For higher flow rates and for fluid films near the nozzle where inertial effects are significant,
systems of coupled equations for both the film thickness and the local flow rate have also been investigated \cite{trifonov1992steady,ruyer2009film,ruyer2008modelling,Ji2020Modeling}. These models include inertia effects based on a weighted-residual integral boundary layer approach by assuming a local velocity profile.
More recently, Ruan et al. \cite{ruan2021liquid} proposed a new framework for liquid films flowing down a fibre using a control-volume approach. Their model expresses the conservation of mass and axial momentum
via a coupled system for the fluid film radius $h(x,t)$ and the mean axial velocity $u(x,t)$,
where the momentum equation is
\begin{equation}\label{momentum}
u_t + a\left( \frac{u^2}{2}\right)_x  + b\, \kappa_x = c\,\frac{[(h^2-1)u_x]_x}{h^2-1} + 1 - \frac{u}{g(h)}
\end{equation}
and the mass conservation equation is
\begin{equation}\label{mass_h}
2 h h_t + a[ u(h^2-1)]_x   = 0 ,
\end{equation}
where the dimensionless parameter $a$ represents the square of the Froude number,
$b$ is the reciprocal of the Bond number,
$c$ represents the ratio of axial viscous to gravitational forces, and $g(h)$ represents the axial velocity profile.
 The film thickness is given by $h(x,t)-1$, and $\kappa$ represents the combined azimuthal and streamwise curvatures of the free surface,
 \begin{equation}\label{curvature}
\kappa = \frac{1}{h(1+h_x^2)^{1/2}}-\left[\frac{h_x}{(1+h_x^2)^{1/2}}\right]_x.
 \end{equation}
 Furthermore,  by taking different forms of $g(h)$, the model \eqref{momentum}--\eqref{mass_h} corresponds to different flow regimes.
 For the high Reynolds number regime, we have the plug flow model with the axial velocity profile
 \begin{equation}\label{plug_model}
     g(h) = h^2-1.
 \end{equation}
This model assumes a uniform velocity in the cross section with a viscous drag force on the fluid. In contrast, for the low Reynolds number case, we assume a fully-developed laminar velocity profile, with 
\begin{equation}\label{laminar_model}
    g(h) = \frac{I(h)}{h^2-1},
\end{equation}
where $I(h)= \tfrac{1}{16}[4 h^4 \ln (h) + (h^2-1)(1-3h^2) ]$.

The control-volume approach has been extensively applied in fluid dynamics problems to analyze mass and momentum balances \cite{sonntag2008fundamentals}. In the context of film flow and one-dimensional hydraulic jumps, Singha at al. \cite{singha2005hydraulic} applied control volume analysis to establish relationships between inlet and outlet variables. Furthermore, the adoption of depth-averaging to replace the velocity component in the flow direction with its average, such as in the work of Bohr et al. \cite{bohr1993shallow}, has a long history in the study of such film flows. The approach of Ruan et al. \cite{ruan2021liquid} combines control volume analysis and depth-averaging for a heuristic derivation of the PDEs describing the liquid film.

\begin{figure}
    \centering 
    \includegraphics[scale = 0.3]{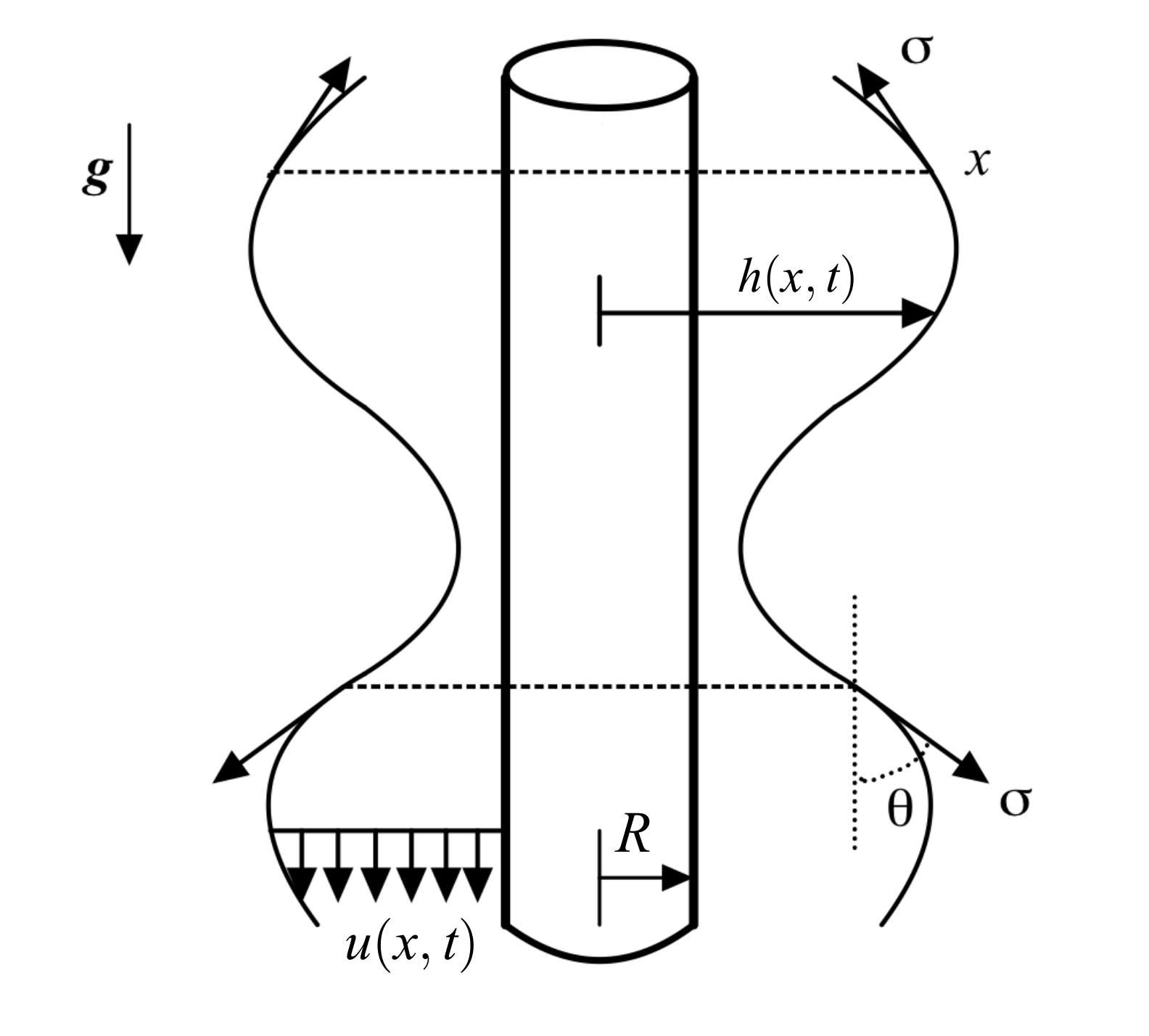} 
    \caption{Schematic plot of a liquid film flowing down a cylindrical fibre. The axial coordinate along the fibre axis is $x$, and the radial distance from that axis is $h(x,t)$. The dimensionless fibre radius $R = 1$.
    }
\label{fig:schematic} 
\end{figure} 

While extensive modelling works have been carried out for falling liquid films, relatively less research \cite{ji2020travelling} have focused on establishing analytical understanding of the developed models.
In this work, we will analytically investigate the coupled equations \eqref{momentum}--\eqref{mass_h}. Energy and entropy estimates will be constructed to establish the existence of weak solutions to the problem. Similar analytical techniques were also applied to other models (for example, see \cite{ji2020travelling,bresch2007existence,bresch2015two,bresch2003existence}).
Theory for the uniqueness of solutions to the system \eqref{momentum}--\eqref{mass_h} is still lacking due to its nonlinear and degenerate nature. However, there have been results regarding solution uniqueness for similar systems involving compressible fluids (see, e.g. \cite{solonnikov1976solvability, desjardins1999survey}).
We will also show the existence of travelling wave solutions to the problem.

The rest of the paper is structured as follows. In section \ref{sec:problem}, we formulate the problem statement. In section \ref{sec:existence}, we show the existence of weak solutions to the problem via energy and entropy estimates. Section \ref{sec:travelling} presents a detailed discussion on travelling wave solutions. Numerical studies are presented in section \ref{sec:numerics} for both the plug flow and the laminar flow cases, followed by concluding remarks in section \ref{sec:conclusion}.

\section{Problem statement}
\label{sec:problem}
We study the following initial boundary value problem:
\begin{equation}\label{r-1}
u_t + a\left( \frac{u^2}{2}\right)_x  + b\, \kappa_x = c\,\frac{[(h^2-1)u_x]_x}{h^2-1} + 1 - \frac{u}{g(h)} \text{ in } Q_T,
\end{equation}
\begin{equation}\label{r-2}
2 h h_t + a[ u(h^2-1) ]_x   = 0 \text{ in } Q_T,
\end{equation}
\begin{equation}\label{r-3}
u \text{ and } h \text{ are }  |\Omega|-\text{periodic},
\end{equation}
\begin{equation}\label{r-4}
u(x,0) = u_0(x), \  h(x,0) = h_0(x),
\end{equation}
where $\Omega \subset \mathbb{R}^1 $ is an open interval, $ Q_T:=\Omega\times (0,T)$, $a,\,b,\,c$ are non-negative constants, and
$$
\kappa = f(h_x) h^{-1} - f^3(h_x)h_{xx}, \  \ f(z) = (1 +z^2)^{-\frac{1}{2}},
$$
$$
\Phi(z) = \frac{1}{f(z)}, \ \Phi'(z) =  z f(z), \  \Phi''(z) =   f^3(z),
$$
$$
g(h)= h^2 -1 \text{ or } g(h)= \frac{I(h)}{h^2 -1},
$$
where
$$
I(h):= \tfrac{1}{16}[4 h^4 \ln (h) + (h^2-1)(1-3h^2) ].
$$
Let
$$
v = h^2 -1.
$$
Then we can rewrite (\ref{r-1}) and (\ref{r-2}) in the following form:
\begin{equation}\label{r-1-0}
u_t + a\left( \frac{u^2}{2}\right)_x  + b\, \kappa_x = c\,\frac{(v \,u_x)_x}{v} + 1 - \frac{u}{g(h)} \quad \text{ in } Q_T,
\end{equation}
\begin{equation}\label{r-2-0}
v_t + a( u \, v )_x   = 0 \quad \text{ in } Q_T.
\end{equation}
Integrating (\ref{r-2-0}) in $Q_t$, we find that $v(x,t)$ satisfies
\begin{equation}\label{mass}
  \int \limits_{\Omega} {v(x,t)\,dx} = \int \limits_{\Omega} {v_0(x)\,dx} := M > 0 \ \ \forall\, t \geqslant 0.
\end{equation}
Let us denote by
\begin{equation}\label{Gdefinition}
G(h) = \int \limits_A^h { y \int \limits_A^y { \frac{s\,ds}{I(s)}  dy  } } \geqslant 0 \text{ for some } A > 1.
\end{equation}
Furthermore, we assume that the initial data $(v_0,u_0)$  satisfy
\begin{equation}\label{rr-incond}
\begin{aligned}
& h_0 \geqslant 1, \text{ i.\,e. }  v_0 := h_0^2 -1 \geqslant 0,  \text{ for all } x \in \bar{\Omega};\\
& \sqrt{v_0} \in H^1(\Omega) ; \ h_0 \Phi(h_{0,x}), \    v_0 u^2_0 \in L^1(\Omega) ;\\
 - & \log (v_0)  \in L^1(\Omega)  \text{ for } g(h) =v, \  G(h_0)  \in L^1(\Omega) \text{ for } g(h) = \tfrac{I(h)}{v} .
\end{aligned}
\end{equation}

\begin{remark}
\label{remark_G}
We note that $G(h)$ defined in \eqref{Gdefinition} has the following asymptotic behavior,  
$$ G(h) \sim \frac{C}{h -1} ~~\text{ as } h \to 1 \text{ for some positive constant }
C .
$$
\noindent Therefore,  $\sqrt{v_0} \in H^1(\Omega)$ and
$G(h_0)  \in L^1(\Omega)$ are satisfied provided $v_0 > 0$ only.
\end{remark}

\begin{definition}\label{Def-weak}
A pair $(h,u)$ is a weak solution to (\ref{r-1-0})--(\ref{r-2-0}) with periodic boundary conditions
and initial conditions $(h_0, u_0)$ if $1 \leqslant h \in C(\bar{Q}_T)$, $v = h^2 - 1$, and $u$ satisfy the regularity properties
\begin{equation}\label{rr-apr}
 \sqrt{v} \in L^{\infty}(0,T; H^1(\Omega)); \ - \log(v) (\text{or } G(h)), \ v u^2 \in L^{\infty}(0,T; L^1(\Omega));
\end{equation}
\begin{equation}\label{rr-apr-2}
h \Phi(h_x) \in   L^{\infty}(0,T; L^1(\Omega));\   h^{-1} f(h_x) \in  L^{1}(Q_T);
\end{equation}
\begin{equation}\label{rr-apr-3}
\chi_{\{|h_x| < \infty \}} \sqrt{h f^3(h_x)} h_{xx}; \ \chi_{\{v > 0\}} \sqrt{v} u_x, \ \sqrt{\tfrac{v}{g(h)}} u \in L^{2}(Q_T),
\end{equation}
and the following holds
$$
\iint \limits_{Q_T} { v  \phi_t \,dx dt} + \int \limits_\Omega {v_0 \phi (x,0)\, dx}
+ {a}\iint \limits_{Q_T} { u v \phi_x  \,dx dt }  = 0,
$$
\begin{align*}
\iint \limits_{Q_T} {u v  \psi_t \,dx dt} &+ \int \limits_\Omega {u_0 v_0 \psi (x,0)\, dx} + \frac{a}{2} \iint \limits_{Q_T} {\chi_{\{v > 0\}}  v u^2  \psi_x \,dx dt} \\
&+b \iint \limits_{Q_T} {( 2 \Phi' (h_{x}) -  \chi_{\{  |h_x| < \infty \}} f^3(h_{x})v_x h_{xx} ) \psi   \,dx dt}  \\
&+b \iint \limits_{Q_T} {( h^{-1}  f(h_{x}) -  \chi_{\{  |h_x| < \infty \}} f^3(h_{x}) h_{xx} ) v \psi_x  \,dx dt}   \\
 &-c \iint \limits_{Q_T} {\chi_{\{v > 0\}} v u_x  \psi_x \,dx dt} +
\iint \limits_{Q_T} { \bigl( v - \tfrac{uv}{g(h)} \bigr)   \psi  \,dx dt} =0
\end{align*}
for all $\phi \in C_c^{\infty}(\bar{Q}_T)$ and $\psi \in  C_c^{\infty}(\bar{Q}_T)$
such that $\phi(x,T) =  \psi(x,T) = 0$.
\end{definition}

We note that the sets $v=0$ and $h_x=\infty$ coincide with the sets $h=1$ and $v_x=\infty$.
Based on Definition \ref{Def-weak}, we will establish the existence of weak solutions to the problem and prove the following theorem:
\begin{theorem}
Let the initial data $(h_0,u_0)$  satisfy \eqref{mass}--\eqref{rr-incond} and $T > 0$. Then there exists
a weak solution {$(h,u)$} in the sense of Definition~\ref{Def-weak}, where {$v=h^2-1$}. Moreover,
the sets $\{ v(.,t) = 0\}$ and $\{ |h_x(.,t)| = \infty\}$ have Lebesgue measure zero for any $t \in [0,T]$
for the plug flow model with $g(h) = v$. In the case of the laminar flow model with $g(h) = \frac{I(h)}{v}$,
we have $v(x,t) > 0$ in $Q_T$.
\end{theorem}

\section{Existence of weak solutions}
\label{sec:existence}

In this section, we will introduce the energy and entropy functionals for the problem and show their estimates in subsections \ref{sec:energy} and \ref{sec:entropy}. The  proof of key results in Lemma \ref{lem-en} and Lemma \ref{lem-entr} follows the work of Kitavtsev et al. \cite{kitavtsev2011weak}.

\subsection{Energy estimate}\label{sec:energy}

Let us denote the energy functional by
$$
\mathcal{E}(t) := \tfrac{1}{2}  \int \limits_{\Omega} {( v u^2 + \tfrac{4 b}{a} h  \Phi(h_x))\,dx} .
$$

\begin{lemma}[\textbf{Energy inequality}]\label{lem-en}
Let 
{$(h,u)$} be {a} solution to the system (\ref{r-1-0})--(\ref{r-2-0}) with
periodic boundary conditions, {where $ 1 < h \in  L^{\infty}(0,T;H^1(\Omega)), u \in L^{\infty} (0,T; L^2(\Omega)) \cap L^{2} (0,T; H^1(\Omega))$}
then $(h,u)$ {satisfies} the following inequality
\begin{equation}\label{e-3}
  \mathcal{E}(T) +
c \iint \limits_{Q_T} {  v u^2_x \,dx dt} + \iint \limits_{Q_T} { \tfrac{u^2 v }{g(h)} \,dx dt} \leqslant
C_0(T) ,
\end{equation}
where {$v = h^2 - 1$},  $C_0(T) = ( \mathcal{E}^{\frac{1}{2}}(0) + \tfrac{\sqrt{2M}}{2}T )^2 $.
\end{lemma}

\begin{proof}[Proof of Lemma~\ref{lem-en}]
Multiplying (\ref{r-1-0}) by $u v$ and integrating over $\Omega$, we have
\begin{multline}\label{e-1}
 \int \limits_{\Omega} {u v u_t\,dx} + a \int \limits_{\Omega} {u v ( \tfrac{u^2}{2})_x \,dx} +
b \int \limits_{\Omega} {u v \kappa_x \,dx} =  
c \int \limits_{\Omega} {u   ( v u_x)_x \,dx} +
\int \limits_{\Omega} {u v ( 1 - \tfrac{u }{g(h)})  \,dx}.
\end{multline}

Since the first two integrals on the left-hand-side of \eqref{e-1} satisfy
\begin{align*}
\int \limits_{\Omega} {u v u_t\,dx} + a \int \limits_{\Omega} {u v ( \tfrac{u^2}{2})_x \,dx} &=
\int \limits_{\Omega} {  v (\tfrac{u^2}{2}) _t\,dx} - a \int \limits_{\Omega} {(u v)_x   \tfrac{u^2}{2}  \,dx}  \\
 &=\int \limits_{\Omega} {  v (\tfrac{u^2}{2}) _t\,dx} + \int \limits_{\Omega} { v_t  \tfrac{u^2}{2}  \,dx} =
 \tfrac{1}{2} \tfrac{d}{dt } \int \limits_{\Omega} {  v  u^2 \,dx} ,
\end{align*}
and the third integral on the left-hand-side of \eqref{e-1} satisfies
\begin{align*}
b \int \limits_{\Omega} {u v \kappa_x \,dx} &= - b \int \limits_{\Omega} {(u v)_x \kappa  \,dx} =
\tfrac{b}{a} \int \limits_{\Omega} {v_t \kappa  \,dx} \\
 &=\tfrac{2 b}{a} \int \limits_{\Omega} { h h_t (f(h_x) h^{-1} - f^3(h_x)h_{xx} )  \,dx}=
 \tfrac{2 b}{a} \int \limits_{\Omega} { (  h_t  f(h_x)   - h h_t \Phi''(h_x)h_{xx} )  \,dx} \\
 &= \tfrac{2 b}{a} \int \limits_{\Omega} { (  h_t  f(h_x)   -  h h_t (\Phi'(h_x))_{x} )  \,dx} =
\tfrac{2 b}{a} \int \limits_{\Omega} { (  h_t  f(h_x)   + ( h h_t)_x  \Phi'(h_x)  )  \,dx}   \\
&=\tfrac{2 b}{a} \int \limits_{\Omega} { (  h_t  f(h_x)   +   h_t h_x  \Phi'(h_x) + h h_{xt}\Phi'(h_x)  )  \,dx} \\
&=\tfrac{2 b}{a} \int \limits_{\Omega} { (  h_t  f(h_x)   +   h_t h^2_x  f(h_x) + h  (\Phi(h_x))_t  )  \,dx} \\
&=\tfrac{2 b}{a} \int \limits_{\Omega} { (  h_t  \Phi(h_x)   + h  (\Phi(h_x))_t  )  \,dx}  =
\tfrac{2 b}{a} \tfrac{d}{dt } \int \limits_{\Omega} {  h  \Phi(h_x) \,dx},
\end{align*}
then from (\ref{e-1}) it follows that
\begin{equation}\label{e-2}
  \tfrac{d}{dt } \mathcal{E}(t) +
c \int \limits_{\Omega} {  v u^2_x \,dx} + \int \limits_{\Omega} { \tfrac{u^2 v }{g(h)} \,dx} =
\int \limits_{\Omega} {u v    \,dx}.
\end{equation}
Taking into account the conservation of mass \eqref{mass} and
\begin{equation}\label{est-1}
\int \limits_{\Omega} {u v \,dx} \leqslant M^{\frac{1}{2}} \Bigl( \int \limits_{\Omega} {  v u^2 \,dx} \Bigr)^{\frac{1}{2}},
\end{equation}
after integrating \eqref{e-2} in time, we obtain the energy estimate (\ref{e-3}).
\end{proof}

\subsection{Entropy estimate}\label{sec:entropy}

We define two entropy-like functionals for the plug flow and laminar flow models separately.
For the plug flow model with $g(h) = v$,
we denote the entropy functional by
$$
\mathcal{S}_1(u,h) := \tfrac{1}{2} \int \limits_{\Omega} { \bigl[   v (u + \tfrac{c}{a} \tfrac{v_x}{v} )^2 + \tfrac{4b}{a}h  \Phi(h_x) +
\tfrac{2c}{a^2} (v - \log (v) ) \bigr]\,dx}.
$$
For the laminar flow model with $g(h) = \frac{I(h)}{v}$, we define the entropy functional as
$$
\mathcal{S}_2(u,h) := \tfrac{1}{2} \int \limits_{\Omega}  \bigl[   v (u + \tfrac{c}{a} \tfrac{v_x}{v} )^2
  + \tfrac{4b}{a} h  \Phi(h_x) +
\tfrac{8c}{a^2} G(h) \bigr]\,dx,
$$
where $G(h)$ is defined in \eqref{Gdefinition}.
Next, we will prove the following entropy estimates for $\mathcal{S}_1(u,h)$ and $\mathcal{S}_2(u,h)$.
\begin{lemma}[\textbf{Entropy inequality}]\label{lem-entr}
Let $(h,u)$ be a solution to the system (\ref{r-1-0})--(\ref{r-2-0}) with
periodic boundary conditions, where $v=h^2-1$, $0 < v \in  L^{\infty} (0,T; H^1(\Omega)) \cap L^{2} (0,T; H^2(\Omega)), u \in L^{\infty} (0,T; L^2(\Omega)) \cap L^{2} (0,T; H^1(\Omega)),$
then $(h,u)$ satisfies the following inequality
\begin{align}
\label{n-6}
\mathcal{S}_i(u,h) +
\tfrac{2bc}{a} \iint \limits_{Q_T} { h f^3(h_x)h^2_{xx} \,dx dt}  +
\iint \limits_{Q_T} {  \tfrac{u^2 v }{g(h)}  \,dx dt}  \\
+\tfrac{2bc}{a} \iint \limits_{Q_T} { h^{-1} f(h_x) \,dx dt} \leqslant
 C_i(T), \ i =1,2, 
 \nonumber
\end{align}
where
$$
C_i(T) := \mathcal{S}_i(u_0,h_0) + \int \limits_{0}^T { c\,C_0(t) + (2M)^{\frac{1}{2}} C_0^{\frac{1}{2}}(t)  \, dt  }.
$$
\end{lemma}

\begin{remark}
Using the estimate $(a+b)^2 \geqslant \epsilon a^2 - \frac{\epsilon}{1-\epsilon} b^2$ for any $\epsilon \in [0,1)$, 
from (\ref{n-6}) and (\ref{e-3}), we deduce that for the plug flow model,
\begin{equation}\label{n-7}
\int \limits_{\Omega} {  \tfrac{v_x^2}{v} \,dx} \leqslant ( \tfrac{a}{c})^2 \bigl[ \tfrac{1}{1-\epsilon}
\int \limits_{\Omega} { v \,u^2 \,dx} + \tfrac{2}{\epsilon}  C_i(T)  \bigr] \leqslant C_3(T) \quad \text{for any } \epsilon \in (0,1),\, i=1,2,
\end{equation}
where
\begin{equation} \label{eq:C3}
C_3(T) := ( \tfrac{a}{c})^2 \bigl[ \tfrac{2}{1-\epsilon}
C_0(T) + \tfrac{2}{\epsilon}  C_i(T)  \bigr].
\end{equation}
From (\ref{n-7}), it follows that $\bigl( v^{\frac{1}{2}}\bigr)_x \in L^{\infty}(0,T; L^2(\Omega))$.
\end{remark}

\begin{proof}[Proof of Lemma~\ref{lem-entr}]
Multiplying (\ref{r-1-0}) by $  v_x $ and integrating over $\Omega$, we have
\begin{equation}\label{n-1}
 \int \limits_{\Omega} {( u_t + a u u_x)v_x \,dx}   +
b \int \limits_{\Omega} {  v_x \kappa_x \,dx} =  
- c \int \limits_{\Omega} { v u_x( \tfrac{v_x}{v})_x \,dx} -
\int \limits_{\Omega} {u   \tfrac{v_x }{g(h)}   \,dx}.
\end{equation}
For the first integral on the left-hand-side of \eqref{n-1}, we have
\begin{align*}
 \int \limits_{\Omega} {( u_t + a u u_x)v_x \,dx} &= \int \limits_{\Omega} {( u_t v_x + a (u v)_x u_x - a\, v u_x^2)\,dx}  \\
&=\int \limits_{\Omega} {( u_t v_x - v_t u_x - a\, v u_x^2)\,dx} = \int \limits_{\Omega} {( u_t v_x + u v_{xt} - a\, v u_x^2)\,dx}\\
&=\tfrac{d}{dt} \int \limits_{\Omega} { u  v_x  \,dx} - a \int \limits_{\Omega} { v u_x^2 \,dx}.
\end{align*}
To handle the first integral on the right-hand-side of \eqref{n-1}, we deduce
\begin{align*}
\tfrac{d}{dt} \int \limits_{\Omega} {\tfrac{v^2_x}{v}  \,dx}   &= 2 \int \limits_{\Omega} {\tfrac{v_x v_{xt}}{v}  \,dx} -
\int \limits_{\Omega} {\tfrac{v_x^2 v_{t}}{v^2}  \,dx} =
- 2 \int \limits_{\Omega} { (\tfrac{v_x  }{v})_x v_t  \,dx} - \int \limits_{\Omega} {\tfrac{v_x^2}{v^2}v_{t}  \,dx} \\
& 
 = 2a \int \limits_{\Omega} { (\tfrac{v_x  }{v})_x (u v)_x dx} +
a\int \limits_{\Omega} {\tfrac{v_x^2}{v^2} (uv)_x \,dx} \\
&=2a \int \limits_{\Omega} { (\tfrac{v_x  }{v})_x (u v)_x dx} - 2a \int \limits_{\Omega} { (\tfrac{v_x  }{v})_x  u v _x dx}=
2a \int \limits_{\Omega} { v u_x (\tfrac{v_x  }{v})_x  dx},
\end{align*}
where the second last step is obtained using integration by parts.

For the second integral on the left-hand-side of \eqref{n-1}, we have
\begin{align*}
b \int \limits_{\Omega} {  v_x \kappa_x \,dx} &= - b \int \limits_{\Omega} {  v_{xx} (f(h_x) h^{-1} - f^3(h_x)h_{xx}) \,dx}    \\
&=- 2b \int \limits_{\Omega} { (h_x^2 + h h_{xx}) (f(h_x) h^{-1} - f^3(h_x)h_{xx}) \,dx}  \\
 &=- 2b \int \limits_{\Omega} { h^{-1}h_x^2 f(h_x) \,dx}
- 2b \int \limits_{\Omega} { f(h_x)(1 - f^2(h_x)h_x^2  )h_{xx} \,dx}  \\
 &\quad + 2b \int \limits_{\Omega} { h f^3(h_x)h^2_{xx} \,dx} \\
& = - 2b \int \limits_{\Omega} { h^{-1}h_x^2 f(h_x) \,dx}
- 2b \int \limits_{\Omega} { f^3(h_x) h_{xx} \,dx} + 2b \int \limits_{\Omega} { h f^3(h_x)h^2_{xx} \,dx} \\
& = - 2b \int \limits_{\Omega} { h^{-1}h_x^2 f(h_x) \,dx} +
 2b \int \limits_{\Omega} { h f^3(h_x)h^2_{xx} \,dx}.
\end{align*}
Then from (\ref{n-1}), it follows that
\begin{multline} \label{n-2}
 \tfrac{d}{dt }\int \limits_{\Omega} { ( u v_x  + \tfrac{c}{2a} \tfrac{v^2_x}{v} )\,dx} + 2b \int \limits_{\Omega} { h f^3(h_x)h^2_{xx} \,dx}     \\
=a \int \limits_{\Omega} {  v u^2_x \,dx} + 2b \int \limits_{\Omega} { h^{-1}h_x^2 f(h_x) \,dx} - \int \limits_{\Omega} {u \tfrac{  v_x }{g(h)} \,dx} .
\end{multline}
Multiplying (\ref{n-2}) by $\frac{c}{a}$ and using (\ref{e-2}), we arrive at
\begin{multline} \label{n-3}
 \tfrac{d}{dt }\int \limits_{\Omega} { ( \tfrac{c}{a} u v_x  + \tfrac{c^2}{2a^2} \tfrac{v^2_x}{v} )\,dx} + \tfrac{d}{dt } \mathcal{E}(t) +
\tfrac{2bc}{a} \int \limits_{\Omega} { h f^3(h_x)h^2_{xx} \,dx}  + \int \limits_{\Omega} { \tfrac{u^2 v }{g(h)} \,dx} \\
=
\tfrac{2bc}{a} \int \limits_{\Omega} { h^{-1}h_x^2 f(h_x) \,dx} - \tfrac{c}{a}\int \limits_{\Omega} {u \tfrac{  v_x }{g(h)} \,dx}
+ \int \limits_{\Omega} {u v    \,dx} .
\end{multline}
Note that
$$
\int \limits_{\Omega} { ( \tfrac{c}{a} u v_x  + \tfrac{c^2}{2a^2} \tfrac{v^2_x}{v} )\,dx} +
\tfrac{1}{2}\int \limits_{\Omega} {  v u^2 \,dx} = \tfrac{1}{2}\int \limits_{\Omega} {  v (u + \tfrac{c}{a} \tfrac{v_x}{v} )^2\,dx},
$$
\begin{multline*}
\tfrac{c}{a}\int \limits_{\Omega} {u \tfrac{  v_x }{g(h)} \,dx} = \tfrac{c}{a}\int \limits_{\Omega} {  \tfrac{ (u v)_x }{g(h)} \,dx}
- \tfrac{c}{a}\int \limits_{\Omega} {  \tfrac{ v}{g(h)} u_x \,dx}    \\
=- \tfrac{c}{a^2}\int \limits_{\Omega} {  \tfrac{ v_t }{g(h)} \,dx}
- \tfrac{c}{a}\int \limits_{\Omega} {  \tfrac{ v}{g(h)} u_x \,dx} = - \tfrac{c}{a^2}\tfrac{d}{dt } \int \limits_{\Omega} { \log (v) \,dx}
\text{ if } g(h) =v.
\end{multline*}
Therefore, for the plug flow model with $g(h) =v$, (\ref{n-3}) has the form
\begin{align} \label{n-4}
 \tfrac{1}{2} \tfrac{d}{dt }\int \limits_{\Omega} &{ \bigl[   v (u + \tfrac{c}{a} \tfrac{v_x}{v} )^2 + \tfrac{4b}{a}h  \Phi(h_x) -
\tfrac{2c}{a^2} \log (v)  \bigr]\,dx}  \\
&+\tfrac{2bc}{a} \int \limits_{\Omega} { h f^3(h_x)h^2_{xx} \,dx}  +
\int \limits_{\Omega} {  u^2   \,dx} + \tfrac{2bc}{a} \int \limits_{\Omega} { h^{-1} f(h_x) \,dx}\nonumber\\
&\quad =\tfrac{2bc}{a} \int \limits_{\Omega} { h^{-1}\Phi(h_x) \,dx} + \int \limits_{\Omega} {u v    \,dx} . \nonumber
\end{align}
If we set $g(h) =\tfrac{I(h)}{v}$ for the laminar flow model, then by multiplying (\ref{r-2-0})  by  $ \int \limits_A^h  {  \frac{y}{I(y)} dy}  $,
we derive
$$
 \int \limits_{\Omega} {u \tfrac{  v_x }{g(h)} \,dx} =
 \tfrac{4 }{a }\tfrac{d}{dt } \int \limits_{\Omega} { G(h) \,dx} ,
\text{ where } G'(h) = h \int \limits_A^h  {  \tfrac{y}{I(y)} dy}.
$$
Using this equality in (\ref{n-3}), we arrive at
\begin{align} \label{n-4-2}
 \tfrac{1}{2} \tfrac{d}{dt }\int \limits_{\Omega} & { \bigl[   v (u + \tfrac{c}{a+\epsilon_0} \tfrac{v_x}{v} )^2
+ \tfrac{4b}{a}h  \Phi(h_x) +
\tfrac{8c}{a^2} G(h)  \bigr]\,dx}  \\
&+\tfrac{2bc}{a} \int \limits_{\Omega} { h f^3(h_x)h^2_{xx} \,dx}  +
\int \limits_{\Omega} { \tfrac{u^2 v^2 }{I(h)} \,dx}  + \tfrac{2bc}{a} \int \limits_{\Omega} { h^{-1} f(h_x) \,dx} \nonumber \\ 
& \quad =\tfrac{2bc}{a} \int \limits_{\Omega} { h^{-1}\Phi(h_x) \,dx} + \int \limits_{\Omega} {u v    \,dx} .\nonumber
\end{align}
Note that $v - \log (v) \geqslant 1 $   and $G(h) = \int \limits_A^h { y \int \limits_A^y { \frac{s\,ds}{I(s)}  dy  } } \geqslant 0$ for all $v \geqslant 0$ and for some $A > 1$. Then, due to (\ref{mass}), (\ref{n-4})
and (\ref{n-4-2}) can be rewritten in the form
\begin{multline} \label{n-5}
  \tfrac{d}{dt } \mathcal{S}_i(u,h) +
\tfrac{2bc}{a} \int \limits_{\Omega} { h f^3(h_x)h^2_{xx} \,dx}  +
\int \limits_{\Omega} {  \tfrac{u^2 v}{g(h)}   \,dx} + \tfrac{2bc}{a} \int \limits_{\Omega} { h^{-1} f(h_x) \,dx} \\
= \tfrac{2bc}{a} \int \limits_{\Omega} { h^{-1}\Phi(h_x) \,dx} + \int \limits_{\Omega} {u v    \,dx}, \ i= 1,2.
\end{multline}
Taking into account (\ref{est-1}),
$$
\int \limits_{\Omega} { h^{-1}\Phi(h_x) \,dx} \leqslant \int \limits_{\Omega} { h \Phi(h_x) \,dx} \text{ for } v \geqslant 0,
$$
and (\ref{e-2}), from (\ref{n-5})  we obtain (\ref{n-6}).
\end{proof}

\subsection{Approximate problem}

For given $\delta > 0$, $\eta > 0$ and $\varepsilon > 0$,  we consider the following approximate system for $(v,u)$, where $v = h^2-1$,
\begin{multline}\label{aprr-1}
(v\, u)_t +  a\,[ ( u v + \eta v^4 v_{xxx}) u ]_x  + b\,v\, \kappa_{x}
\\
= c\, (v\,u_x)_x  + v \bigl( 1 - \tfrac{u}{g(h)}\bigr)
-  \delta  u_{ xxxx} + \varepsilon \, a  [p(v)]_x  
- \varepsilon \, a \,v v_{xxxxx} ,
\end{multline}
\begin{equation}\label{aprr-2}
v_{t} + a( u\,v  )_x   = - a\, \eta \left(v^4 v_{xxx}\right)_x  ,
\end{equation}
\begin{equation}\label{aprr-3}
u \text{ and } v \text{ are }  |\Omega|-\text{periodic},
\end{equation}
\begin{equation}\label{aprr-4}
u(x,0) = u_{\varepsilon \eta,0}(x), \  v (x,0) = v_{\varepsilon \eta,0}(x), 
\end{equation}
where $u_{\varepsilon \eta,0} \in H^{1}( \Omega )$ and $0 < v_{\varepsilon \eta,0} \in H^{2}( \Omega )$ such that
$$
u_{\varepsilon,0}(x) \to u_{0}(x) \text{ strongly in } L^2(\Omega), \  v_{\varepsilon,0}(x) \geqslant  v_{0}(x) + \varepsilon^{\theta},
\  \theta \in (0,\tfrac{1}{2}) ,
$$
$$
v_{\varepsilon,0}(x) \to v_{0}(x) \text{ strongly in } W_1^1(\Omega) \cap C(\bar{\Omega}),
$$
$$
 \varepsilon^{\frac{1}{2}}v_{\varepsilon, 0xx }(x) \to 0 \text{ strongly in } L^2(\Omega),
$$
and
$$
p(z) = \tfrac{1}{2} z^{-2}  , \ g(h)  = v  \text{ or }  g(h)  = \tfrac{|I(h)|}{v } .
$$
Here, the term $\eta (v^4 v_{xxx})_x$ is a thin-film type regularization term that yields uniform parabolicity and provides a positive approximation of the solution. The term $\epsilon a [p(v)]_x$ leads to a positive approximation during the limit processes 
as $\epsilon \to 0$. The term $\epsilon a v v_{xxxxx}$ is crucial for controlling additional higher-order
terms that arise in the entropy inequality.
This regularization approach is in the spirit of the construction of weak solutions of the Burger's equation using nonlinear viscous terms and is also comparable to the approximation method employed in \cite{kitavtsev2011weak,bresch2006construction}. Specifically, the incorporation of the thin-film term $\eta (v^4 v_{xxx})_x$ in our approximate system results in the reduction of the seventh-order derivative term of the form $\epsilon h h^{(7)}$, as required in the approximation system considered in \cite{bresch2006construction}, to a fifth-order derivative term $\epsilon a v v_{xxxxx}$ in \eqref{aprr-1}.

\begin{remark}
    We note that in the laminar flow case, due to the strong positivity of $v$, i.e. $v > 0$, from the boundedness of the entropy, the approximate system \eqref{aprr-1}--\eqref{aprr-2} can be simplified by dropping the regularization term $\epsilon(p(v))_x $.
\end{remark}

\begin{lemma}\label{lem-r}
Let $u \in L^2(0,T; H^2(\Omega)) $ be a periodic function. For any $0 < v_{\eta \varepsilon,0}(x) \in H^2 ( \Omega )$,
the problem (\ref{aprr-2})---(\ref{aprr-3}) has a unique weak positive solution $v \in C_{x,t}^{\frac{3}{2},\frac{3}{8} } (\bar{Q}_T)$
such that
\begin{equation}\label{bf-0}
v_{ t} \in L^2(Q_T), \ v  \in L^{\infty}(0,T; H^2(\Omega)), \ v  \in L^{2}(0,T; H^4(\Omega)),
\end{equation}
$ \int \limits_{\Omega} { v \,dx} =  \int \limits_{\Omega} { v_{\eta \varepsilon,0} \,dx}
 := M_{\eta\varepsilon} > 0 $, and $v$ satisfies (\ref{aprr-2}) a.\,e. in $Q_T$. Moreover,
 there exists a constant $C > 1$ depending on $\eta$ such that
$$
\tfrac{1}{C} \leqslant v  \leqslant C.
$$
\end{lemma}

\begin{proof}[Proof of Lemma~\ref{lem-r}]
The main line of proof follows the approach in \cite{B8}. We omit details and restrict the discussion only to the key elements of the proof.

First of all, we approximate (\ref{aprr-2}) by
\begin{equation}\label{aprr-2-000}
v_{t} + a( [u]_{\alpha} v  )_x   = - a\, \eta ((v^4 +\beta) v_{xxx})_x  ,
\end{equation}
where $\beta > 0$ and $[u]_{\alpha}$ denotes a smooth approximation of $u$ such that
$$
[u]_{\alpha} \to u \text{ strongly in } L^2(0,T; H^2(\Omega)) \text{ as } \alpha \to 0.
$$
We also approximate  $v_{\eta \varepsilon,0}$ in the $H^2$-norm by $C^{4+\gamma}$ functions
$v_{\beta  \eta \varepsilon,0}$, satisfying (\ref{aprr-3}), and replace (\ref{aprr-4}) by
\begin{equation}\label{aprr-4-000}
v (x,0) = v_{\beta\varepsilon \eta,0}(x).
\end{equation}
Using the parabolic Schauder estimates from \cite{Sol},
one can generalise \cite[Theorem 6.3,
p. 302]{eidel1969parabolic} and show that the problem (\ref{aprr-2-000})--(\ref{aprr-4-000}) has a unique
classical solution $v_{\beta\alpha} \in C^{4+\gamma, 1+ \frac{\gamma}{4}}_{x,t}(\Omega \times [0,\tau_{\beta\alpha}])$ for some
$\tau_{\beta\alpha} > 0$; that is, the fourth-order spatial derivative $v_{\beta\alpha,xxxx} \in C_{x,t}^{\gamma, 1+\gamma/4}(\Omega \times [0,\tau_{\beta\alpha}])$, and the time derivative $v_{\beta\alpha,t} \in C_{x,t}^{4+\gamma,\gamma/4}(\Omega \times [0,\tau_{\beta\alpha}])$. 

Next, for simplicity, we will replace $v_{\beta \alpha}$ by $v$. 
Multiplying (\ref{aprr-2-000}) by $-v_{xx}$ and integrating by parts, we have
\begin{align*}
 \tfrac{1}{2} \tfrac{d}{dt} \int \limits_{\Omega} { v^2_{x } \,dx} & +
\eta a \int \limits_{\Omega} { (v^4 +\beta) v^2_{xxx} \,dx} = a \int \limits_{\Omega} { ( [u]_{\alpha} v)_x v_{xx} \,dx} \\
& = - a \int \limits_{\Omega} { ([u]_{\alpha})_{xx}  v v_{x} \,dx} -
    \tfrac{3a}{2} \int \limits_{\Omega} {([u]_{\alpha})_x v^2_{x} \,dx} \\
& \leqslant  a  \mathop {\sup} \limits_{\Omega} |v| \Bigl(  \int \limits_{\Omega} { ([u]_{\alpha})^2_{xx} \,dx}  \Bigr)^{\frac{1}{2}}
 \Bigl(  \int \limits_{\Omega} {v^2_{x } \,dx}  \Bigr)^{\frac{1}{2}} +
  \tfrac{3a}{2} \mathop {\sup} \limits_{\Omega} |([u]_{\alpha})_x| \int \limits_{\Omega} {  v^2_{x} \,dx}  \\
& \leqslant \Bigl(  \int \limits_{\Omega} { ([u]_{\alpha})^2_{xx} \,dx}  \Bigr)^{\frac{1}{2}} \Bigl[ \tfrac{5a |\Omega|^{\frac{1}{2}}}{2}
\int \limits_{\Omega} {v^2_{x } \,dx} + \tfrac{M_{\beta\eta\varepsilon}}{|\Omega|}
\Bigl(  \int \limits_{\Omega} {v^2_{x } \,dx}  \Bigr)^{\frac{1}{2}}   \Bigr],
\end{align*}
whence from the Gronwall inequality we obtain
\begin{align*}
 \int \limits_{\Omega} { v^2_{x } \,dx} 
 \leqslant &  \Bigl( \| v_{\beta \eta \varepsilon,0x} \|_2 +
\tfrac{M_{\beta\eta\varepsilon}}{|\Omega|} \int \limits_0^T {\|([u]_{\alpha})_{xx}\|_2
e^{ - \frac{5a|\Omega|^{\frac{1}{2}} }{2} \int \limits_0^t {\|([u]_{\alpha})_{xx}\|_2 \, ds} } \, dt}    \Bigr)^2 \\
&\times e^{ 5 a |\Omega|^{\frac{1}{2}}   \int \limits_0^T {\|([u]_{\alpha})_{xx}\|_2 \, dt} },\nonumber
\end{align*}
which leads to
\begin{equation}\label{bf-1}
 \int \limits_{\Omega} { v^2_{x } \,dx} + \eta a \iint \limits_{Q_T} { (v^4 +\beta) v^2_{xxx} \,dx dt}
 \leqslant   C(T).
\end{equation}
Due to (\ref{bf-1}), we deduce that $\| v_{\beta}\|_{C_{x,t}^{\frac{1}{2}, \frac{1}{8}}(\bar{Q}_T)}$
is uniformly bounded with respect to $\alpha$, $\beta$ and  $\tau_{\beta\alpha}$.  For any fixed
values of $\beta$ and $\alpha$, by \cite[Theorem 9.3, p. 316]{eidel1969parabolic}, we can  extend the solution
$v_{\beta}$ step-by-step to all of $Q_T$ for any $T > 0$.

Let us denote by $\mathcal{G}_{\beta}(z) \geqslant 0$ such that
$$
\mathcal{G}_{\beta}(z) = \int \limits_1^z { \int \limits_1^y { \tfrac{ds dy}{s^4 + \beta}}}, \ \ \mathcal{G}''_{\beta}(z) = \tfrac{1}{z^4 + \beta}  .
$$
Multiplying (\ref{aprr-2}) by $\mathcal{G}'_{\beta}(v)$ and integrating by parts, we arrive at
\begin{align*}
 \tfrac{d}{dt} \int \limits_{\Omega} { \mathcal{G}_{\beta}(v) \,dx} +
\eta a \int \limits_{\Omega} {   v^2_{xx } \,dx} &= - a  \int \limits_{\Omega} { ([u]_{\alpha})_{x }(v\mathcal{G}'_{\beta}(v) - \mathcal{G}_{\beta}(v) ) \,dx}
 \\
&\leqslant  a C \mathop {\sup} \limits_{\Omega} |([u]_{\alpha})_x| \int \limits_{\Omega} { \mathcal{G}_{\beta}(v) \,dx},
\end{align*}
whence
\begin{equation}\label{bf-2}
\int \limits_{\Omega} { \mathcal{G}_{\beta}(v) \,dx} +
\eta a \iint \limits_{Q_T} {   v^2_{xx } \,dx dt} \leqslant
e^{  a C  \int \limits_0^T{ \mathop {\sup} \limits_{\Omega} |([u]_{\alpha})_x| \,dt} } \int \limits_{\Omega} { \mathcal{G}_{\beta}(v_{\beta\eta \varepsilon,0}) \,dx}.
\end{equation}
Due to (\ref{bf-1}) and (\ref{bf-2}), similar to the proof of \cite[Theorem 4.1, p.190]{B8},
after taking
$\beta \to 0$, we obtain the global existence of a unique positivity classical solution $v_{0\alpha}$ for
any $\alpha > 0$. Moreover, we have $\frac{1}{C}  \leqslant v_{0\alpha} \leqslant   C  < \infty$, where
$C > 1$ is independent of $\alpha$.

For the limit process $\alpha \to 0$, we need the following a priori estimate.
Multiplying (\ref{aprr-2-000}) with $\beta = 0$ by $ v_{xxxx}$ and integrating by parts, we have
\begin{align*}
 \tfrac{1}{2} \tfrac{d}{dt} &\int \limits_{\Omega} { v^2_{xx} \,dx} +
\eta a \int \limits_{\Omega} { v^4 v^2_{xxxx} \,dx} =
- a \int \limits_{\Omega} { ([u]_{\alpha} v)_x v_{xxxx} \,dx} - 
4\eta a \int \limits_{\Omega} { v^3 v_x v_{xxx} v_{xxxx} \,dx} \\
& \leqslant
a \Bigl(  \int \limits_{\Omega} {  v^4 v^2_{xxxx} \,dx}  \Bigr)^{\frac{1}{2}} \Bigl(  \int \limits_{\Omega} {\tfrac{([u]_{\alpha}  v)^2_x} {v^4}\,dx}  \Bigr)^{\frac{1}{2}} + 
4\eta a \Bigl(  \int \limits_{\Omega} {  v^4 v^2_{xxxx} \,dx}  \Bigr)^{\frac{1}{2}}
\Bigl(  \int \limits_{\Omega} {  v^2 v_x^2 v^2_{xxx} \,dx}  \Bigr)^{\frac{1}{2}} \\ 
& \leqslant 
 \tfrac{\eta a}{2} \int \limits_{\Omega} { v^4 v^2_{xxxx} \,dx} +
 \tfrac{a}{\eta}\int \limits_{\Omega} {\tfrac{([u]_{\alpha}  v)^2_x} {v^4}\,dx}  +
 16 a \eta  \int \limits_{\Omega} {  v^2 v_x^2 v^2_{xxx} \,dx} \\
 & \leqslant 
\tfrac{\eta a}{2} \int \limits_{\Omega} { v^4 v^2_{xxxx} \,dx} +
\tfrac{2a} {\eta} \bigl( \| v^{-1}  \|^2_{\infty} \| ([u]_{\alpha} )_x \|^2_2 + \| v^{-1}  \|^4_{\infty}
\| v_x \|^2_{\infty} \| [u]_{\alpha}  \|^2_2 \bigr)  \\
&\quad +16 a \eta   \| v^{-1}  \|^2_{\infty} \| v_x \|^2_{\infty} \int \limits_{\Omega} {  v^4 v^2_{xxx} \,dx} ,
\end{align*}
whence
\begin{multline*}
 \tfrac{d}{dt}\| v_{xx} \|^2_{2}  +
\eta a \int \limits_{\Omega} { v^4 v^2_{xxxx} \,dx} \leqslant   a C \, \eta^{-1} \|[u]_{\alpha} \|^2_{H^1} \\
+  
a C \bigl( \eta^{-1} \|[u]_{\alpha} \|^2_{H^1} + \eta \int \limits_{\Omega} {  v^4 v^2_{xxx} \,dx}   \bigr)\| v_{xx} \|^2_{2}.
\end{multline*}
Integrating this inequality in time, we have
\begin{align}\label{bf-3}
 \int \limits_{\Omega}& { v^2_{xx} \,dx} + \eta a \iint \limits_{Q_T} { v^4 v^2_{xxxx} \,dx dt} \\
 & \leqslant \Bigl( \| v_{\eta \varepsilon,0xx} \|^2_2 +    
a C \, \eta^{-1}  \int \limits_0^T {  \|[u]_{\alpha} \|^2_{H^1}
e^{ -  a C \int \limits_0^t {\bigl( \eta^{-1} \|[u]_{\alpha} \|^2_{H^1} + \eta \int \limits_{\Omega} {  v^4 v^2_{xxx} \,dx}   \bigr) \, ds} } \, dt}    \Bigr)
\nonumber \\
 & \qquad  \times e^{a C \int \limits_0^t {\bigl( \eta^{-1} \|[u]_{\alpha} \|^2_{H^1} + \eta \int \limits_{\Omega} {  v^4 v^2_{xxx} \,dx}   \bigr) \, ds} }. \nonumber 
\end{align}
By  (\ref{bf-3}) and (\ref{bf-1}), $\| v_{0\alpha}\|_{C_{x,t}^{\frac{3}{2}, \frac{3}{8}}(\bar{Q}_T)}$
is uniformly bounded with respect to $\alpha$. This uniform bound follows from
$v_{0\alpha} \in L^{\infty}(0,T; H^2(\Omega))$ and $v_{0\alpha, t} \in L^{2}(Q_T)$ (see \cite[Lemma 7.19, p.~175]{vazquez2007porous}).
Taking $\alpha \to 0$, it completes the proof.
\end{proof}

For the given $\delta > 0$, $\eta > 0$ and $\varepsilon > 0$, equation (\ref{aprr-1}) is uniformly parabolic with respect to $u$
for any $v$ is from Lemma~\ref{lem-r}. By using Faedo-Galerkin approximation (see, e.\,g., \cite{bresch2015two}), the system (\ref{aprr-1})--(\ref{aprr-2})  with periodic boundary conditions has a local in time weak solution
$(v,u) := (v_{ \delta\eta\varepsilon}, u_{ \delta\eta\varepsilon})$.
Next, we establish a priori estimates which guarantee  the global in time solvability.

\begin{lemma}[a priori estimates]\label{apr-est-eda}
For fixed and positive constants $\delta > 0$, $\eta > 0$, $\varepsilon > 0$,  and $T > 0$,  there exists a weak solution
$(v_{\delta\eta\varepsilon}, u_{\delta\eta\varepsilon})$ to
the  problem  (\ref{aprr-1})--(\ref{aprr-4}) in the following sense
\begin{align}\label{ident-1}
\iint \limits_{Q_T} &{u v  \psi_t \,dx dt} + 
 \int \limits_\Omega {u_{\varepsilon\eta,0} v_{\varepsilon\eta,0} \psi (x,0)\, dx} +
 a \iint \limits_{Q_T} { (u v + \eta v^4 v_{xxx}) u  \psi_x \,dx dt} \\
&+b \iint \limits_{Q_T} { \kappa  v_x  \psi  \,dx dt} + b \iint \limits_{Q_T} { \kappa  v \psi_x  \,dx dt}   - c \iint \limits_{Q_T} { v u_x  \psi_x \,dx dt} \nonumber\\
& + \iint \limits_{Q_T} { \bigl( v - \tfrac{uv}{g(h)} \bigr)   \psi  \,dx dt}
- \delta \iint \limits_{Q_T} {u_{xx} \psi_{xx}  \,dx dt} - \varepsilon a
\iint \limits_{Q_T} { p(v) \psi_x  \,dx dt} \nonumber\\
& - \varepsilon a
\iint \limits_{Q_T} {   v_{xxx} ( v_{xx} \psi + 2 v_x \psi_x + v \psi_{xx} )   \,dx dt}  = 0, \nonumber
\end{align}
\begin{multline}\label{ident-2}
 \iint \limits_{Q_T} { v  \phi_t \,dx dt} + \int \limits_\Omega {v_{\varepsilon\eta,0} \phi (x,0)\, dx}
+   
a \iint \limits_{Q_T} { u v \phi_x  \,dx dt }  + a \eta
\iint \limits_{Q_T} { v^4 v_{xxx} \phi_x  \,dx dt }= 0
\end{multline}
for all $\phi \in C_c^{\infty}(\bar{Q}_T)$ and $\psi \in  C_c^{\infty}(\bar{Q}_T)$
such that $\phi(x,T) =  \psi(x,T) = 0$.

Moreover, there exists a positive constant $C > 0$ depending only on $a, \, b, \, c, \, T, \, \mathcal{E}(0)$, and
$ \mathcal{S}_i(u_0, v_0)$   such that the following terms are bounded
by $C$ in respective norms
\begin{equation}\label{r-rr-apr}
 \sqrt{v} \in L^{\infty}(0,T; H^1(\Omega)), \ \sqrt{v} u \in L^{\infty}(0,T; L^2(\Omega)),
\end{equation}
\begin{equation}\label{r-rr-apr-2}
- \log (v) (\text{or } G(h)), \  h \Phi(h_x) \in   L^{\infty}(0,T; L^1(\Omega)),\  h^{-1} f(h_x) \in  L^{1}(Q_T),
\end{equation}
\begin{equation}\label{r-rr-apr-3}
\sqrt{h f^3(h_x)} h_{xx}, \ \sqrt{v} u_x, \ \sqrt{\tfrac{v}{g(h)}} u    \in L^{2}(Q_T),
\end{equation}
\begin{equation}\label{apr-nn-1}
\varepsilon^{\frac{1}{2}} v^{-1}, \  \varepsilon^{\frac{1}{2}} v_{xx} \in  L^{\infty}(0,T; L^2(\Omega)),
\end{equation}
\begin{equation}\label{apr-nn-2}
\delta^{\frac{1}{2}} u_{xx}, \  (\varepsilon \eta)^{\frac{1}{2}} v^2 v_{xxxx}, \
\varepsilon^{\frac{1}{2}}   v_{xxx },  \ \varepsilon^{\frac{1}{2}} (v^{-1})_x \in L^2(Q_T),
\end{equation}
and
\begin{equation}\label{r-rr-apr-4}
0 < \varepsilon  \leqslant v(x,t) \leqslant C  \text{ for all } (x,t) \in Q_T.
\end{equation}
\end{lemma}

\begin{proof}[Proof of Lemma~\ref{apr-est-eda}]
Let us denote by
$$
\mathcal{E}_{\varepsilon}(t) := \tfrac{1}{2}  \int \limits_{\Omega} {( v \,u^2 + \tfrac{2 b}{a} h  \Phi(h_x)
+  \tfrac{\varepsilon}{3} v^{-2}  +   \varepsilon v^2_{xx})\,dx}.
$$
Multiplying (\ref{aprr-1}) by $u  $ and integrating over $\Omega$, we have
\begin{multline} \label{nnn-1}
  \tfrac{d}{dt } \mathcal{E}_{\varepsilon}(t) + \delta \int \limits_{\Omega} {    u^2_{xx} \,dx}
+ \varepsilon \eta a \int \limits_{\Omega} {  v^4 v^2_{xxxx} \,dx} + 
  \varepsilon \eta a \int \limits_{\Omega} {   v^2_{xx} \,dx} + c \int \limits_{\Omega} {  v  u^2_x \,dx} +
 \int \limits_{\Omega} { \tfrac{u^2 v  }{g_\varepsilon(h)} \,dx} \\
 =
\int \limits_{\Omega} {u v    \,dx}  - 
4 \varepsilon \eta a \int \limits_{\Omega} { v^3 v_x v_{xxx} v_{xxxx} \,dx} -
b \eta  \int \limits_{\Omega} { (v^4  v_{xxx})_x \kappa \,dx} .
\end{multline}
Note that
$$
\int \limits_{\Omega} {u v    \,dx} \leqslant M^{\frac{1}{2}} \Bigl( \int \limits_{\Omega} {  v u^2 \,dx} \Bigr)^{\frac{1}{2}}.
$$
We will use the following estimates
$$
\|v \|_{\infty}  \leqslant C\,\|  v_{xx} \|_2 + \tfrac{M}{|\Omega|}, \ \|v_x \|_{\infty}  \leqslant C\,\|  v_{xx} \|_2,
\  \|  v_{xxx} \|_2 \leqslant C\,\|  v_{xxxx} \|_2^{\frac{1}{2}} \|  v_{xx} \|_2^{\frac{1}{2}},
$$
and
\begin{equation}\label{inf}
\| v^{-1}\|_{\infty} \leqslant C\, \bigl( \| v^{-1} \|^{\frac{3}{2}}_{2} ( \| v_{xx} \|^2_{2} + \nu )^{\frac{1}{4}}
+ ( \| v_{xx} \|^2_{2} + \nu )^{- \frac{1}{2}}  \bigr) \ \forall \, \nu \geqslant 0.
\end{equation}
For the proof of (\ref{inf}), see \cite[Lemma 3.2, p.807]{CD}.
From \eqref{inf}, we deduce that
\begin{align*}
\int \limits_{\Omega}& { v^3 v_x v_{xxx} v_{xxxx} \,dx}  \leqslant
 \| v^2  v_{xxxx} \|_2 \| v v_x  v_{xxx} \|_2  \\ 
 &\leqslant  
\| v^2  v_{xxxx} \|_2 \|v \|_{\infty} \| v_x\|_{\infty} \|  v_{xxx} \|_2  \nonumber\\
& \leqslant C\, \| v^2  v_{xxxx} \|_2 \|v \|_{\infty} \| v_x\|_{\infty} \|  v_{xxxx} \|_2^{\frac{1}{2}} \|  v_{xx} \|_2^{\frac{1}{2}} \nonumber\\
&\leqslant 
C\, \| v^2  v_{xxxx} \|^{\frac{3}{2}}_2 \| v^{-1} \|_{\infty}
\bigl( \tfrac{M}{|\Omega|} + \|  v_{xx} \|_2 \bigr) \|  v_{xx} \|_2^{\frac{3}{2}} \nonumber\\
& \leqslant 
C\, \| v^2  v_{xxxx} \|^{\frac{3}{2}}_2 \bigl( \| v^{-1} \|^{\frac{3}{2}}_{2} ( \| v_{xx} \|^2_{2} + 1 )^{\frac{1}{4}}
+ 1  \bigr)
\bigl( \tfrac{M}{|\Omega|} + \|  v_{xx} \|_2 \bigr) \|  v_{xx} \|_2^{\frac{3}{2}} ,
\end{align*}
and
\begin{align*}
  \int \limits_{\Omega}& { (v^4  v_{xxx})_x \kappa \,dx} = \int \limits_{\Omega} { v^4  v_{xxxx}  \kappa \,dx} +
  4 \int \limits_{\Omega} {  v^3 v_x  v_{xxx}  \kappa \,dx} \\
 & \leqslant 
\tfrac{\varepsilon}{4}  \int \limits_{\Omega} { v^4  v^2_{xxxx}   \,dx}  +
\tfrac{2}{\varepsilon}  \int \limits_{\Omega} { v^4 \kappa^2 \,dx}  +
\tfrac{C}{\varepsilon} \|v^{-1}\|^4_{\infty}  \int \limits_{\Omega} { v^6 v_x^2 \kappa^2 \,dx} \\
& \leqslant 
\tfrac{\varepsilon}{4}  \int \limits_{\Omega} { v^4  v^2_{xxxx}   \,dx} +
\tfrac{C}{\varepsilon}  \int \limits_{\Omega} {( v^3 + v^3 v^2_{xx} ) \,dx} +
\tfrac{C}{\varepsilon} \|v^{-1}\|^4_{\infty} \int \limits_{\Omega} {( v^6 + v^6 v^2_{xx} ) \,dx} \\
&\leqslant 
\tfrac{\varepsilon}{4}  \int \limits_{\Omega} { v^4  v^2_{xxxx}   \,dx} +
\tfrac{C}{\varepsilon} (1 + \|v_{xx}\|_2^2 ) \bigl[
(\tfrac{M}{|\Omega|})^3 + \|v_{xx}\|_2^3  +  \|v^{-1}\|^4_{\infty} ((\tfrac{M}{|\Omega|})^6 + \|v_{xx}\|_2^6) \bigr] .
\end{align*}
By these estimates, from (\ref{nnn-1}) we have
\begin{multline} \label{nnn-1-000}
  \tfrac{d}{dt } \mathcal{E}_{\varepsilon}(t) + \delta \int \limits_{\Omega} {    u^2_{xx} \,dx}
+ \tfrac{\varepsilon \eta a }{2}\int \limits_{\Omega} {  v^4 v^2_{xxxx} \,dx} + \varepsilon \eta a \int \limits_{\Omega} {   v^2_{xx} \,dx} +
c \int \limits_{\Omega} {  v  u^2_x \,dx} + 
 \int \limits_{\Omega} { \tfrac{u^2 v  }{g_\varepsilon(h)} \,dx} \\
 \leqslant
C\,  \mathcal{E}^{\frac{1}{2}}_{\varepsilon}(t)+
C\, b \eta \varepsilon^{-9} \mathcal{E}^{8}_{\varepsilon}(t)  + C\, a \eta \varepsilon^{-8} \mathcal{E}^{9}_{\varepsilon}(t) ,
\end{multline}
whence by the nonlinear Gronwall inequality, we have
\begin{multline} \label{nnn-1-001}
\mathcal{E}_{\varepsilon}(T) + \delta \iint \limits_{Q_T} {    u^2_{xx} \,dxdt}
+ \tfrac{\varepsilon \eta a}{2} \iint \limits_{Q_T} {  v^4 v^2_{xxx} \,dxdt } 
\\
+\varepsilon \eta a \iint \limits_{Q_T} {   v^2_{xx} \,dx dt} + c \iint \limits_{Q_T} {  v  u^2_x \,dxdt} +  \iint \limits_{Q_T} { \tfrac{u^2 v  }{g(h)} \,dx dt}  \leqslant C(T)
\end{multline}
for all 
$T \leqslant T_{\eta  } : =
{[16 C {\max\{1,a,b\}}\mathcal{E}^{8}_{\varepsilon}(0) \eta \varepsilon^{-8}]^{-1}} \to +\infty$
as $\eta \to 0$. 
In particular, from (\ref{nnn-1-001}), due to (\ref{inf}), we arrive at
\begin{equation}\label{r-lb}
\| v^{-1}\|_{\infty} \leqslant C \, \varepsilon^{- 1},
\end{equation}
whence $\mathop {\inf} \limits_{\Omega} v \geqslant C \, \varepsilon $.

Next, we consider the plug flow case $g(h) =v$ only, and the laminar flow case can be addressed
in a similar manner. Multiplying (\ref{aprr-1}) by $ \frac{v_x}{v}  $ and integrating over $\Omega$, we obtain that
\begin{align} \label{n-4-001}
  \tfrac{d}{dt }& \mathcal{S}_{1,\varepsilon}(u,v) +
\tfrac{2bc}{a} \int \limits_{\Omega} { h f^3(h_x)h^2_{xx} \,dx}  +
\int \limits_{\Omega} {  u^2   \,dx} + \tfrac{2bc}{a} \int \limits_{\Omega} { h^{-1} f(h_x) \,dx} \\ 
& \quad + \delta \int \limits_{\Omega} {    u^2_{xx} \,dx}
+ \varepsilon a  \int \limits_{\Omega} { v^2_{xxx } \,dx} + \tfrac{ \varepsilon a}{3} \int \limits_{\Omega} {\tfrac{ v^2_{x }}{v^{4}} \,dx} \nonumber\\
&= \tfrac{2bc}{a} \int \limits_{\Omega} { h^{-1}\Phi(h_x) \,dx} + \int \limits_{\Omega} {u v \,dx}
- \delta  \int \limits_{\Omega} { \bigl( \tfrac{v_x}{v} \bigr)_{xx}u_{xx} \,dx} + \eta \, c \int \limits_{\Omega} { v u_{x} (v^3 v_{xxx})_x \,dx} \nonumber\\ 
& \quad +
 \eta \tfrac{c^2 }{a} \int \limits_{\Omega} {  \bigl( \tfrac{v_x}{v} \bigr)_{x} v (v^3 v_{xxx})_x  \,dx}
 + \eta \tfrac{ c }{a} \int \limits_{\Omega} {  \tfrac{ (v^4 v_{xxx})_x}{g(h)} \,dx} ,\nonumber
\end{align}
where
\begin{multline*}
\mathcal{S}_{1,\varepsilon}(u,v) := \tfrac{1}{2} \int \limits_{\Omega}   \bigl[   v (u + \tfrac{c}{a} \tfrac{v_x}{v} )^2 + \tfrac{4b}{a}h  \Phi(h_x) 
+ \tfrac{2c}{a^2} (v - \log (v) ) +    
2 \varepsilon v^{-2} + \varepsilon v^2_{xx } \bigr]\,dx .
\end{multline*}
Using the equality
\begin{equation}\label{intr-tr}
\left( \frac{v_x}{v} \right)_{xx} = \frac{v_{xxx}}{v} - 3 \frac{v_x v_{xx}}{v^2} + 2\frac{v^3_x}{v^3},
\end{equation}
and the following estimates
$$
\| v_x \|_{6} \leqslant C \| v_{xx } \|^{\frac{1}{3}}_{ 2}  \| v_{x} \|^{\frac{2}{3}}_{ 2}, \
\| v_{xx} \|_{\infty} \leqslant C \| v_{xxx} \|_{ 2} ,
$$
we find that
\begin{align*}
\delta  \int \limits_{\Omega} & { \bigl( \tfrac{v_x}{v} \bigr)_{xx}u_{xx} \,dx} \leqslant
\delta^{\frac{1}{2}} \|\delta^{\frac{1}{2}} u_{xx}\|_2
\bigl( \| \tfrac{v_{xxx}}{v} \|_2 + 3 \| \tfrac{v_x v_{xx}}{v^2} \|_2 + 2 \| \tfrac{v_{x}}{v} \|^3_6 \bigr) \\
&\leqslant  \delta^{\frac{1}{2}} \|\delta^{\frac{1}{2}} u_{xx}\|_2
\bigl( \| v^{-1}\|_{\infty}  \|  v_{xxx} \|_2 + 3\| v^{-1}\|^2_{\infty} \|  v_x  \|_2
\| v_{xx} \|_{\infty} + 2 \| v^{-1}\|^3_{\infty} \|  v_{x}\|^3_6 \bigr) \\
& \leqslant 
C\,\delta^{\frac{1}{2}} \|\delta^{\frac{1}{2}} u_{xx}\|_2
\bigl( \| v^{-1}\|_{\infty}  \|  v_{xxx} \|_2 +  \| v^{-1}\|^2_{\infty}\|  v_x  \|_2 \|  v_{xxx} \|_2  \\
& \quad + 
 \| v^{-1}\|^3_{\infty}\|  v_x  \|^2_2 \|  v_{xxx} \|_2 \bigr) \leqslant
 C\,\delta^{\frac{1}{2}} \varepsilon^{- 4} \|\delta^{\frac{1}{2}} u_{xx}\|_2 \|  v_{xxx} \|_2 \\
& \leqslant 
C\,\delta^{\frac{1}{2}} \varepsilon^{- 4} \|\delta^{\frac{1}{2}} u_{xx}\|_2 \|  v_{xxxx} \|_2 \leqslant C\,\delta^{\frac{1}{2}} \varepsilon^{- \frac{13}{2}} \eta^{-\frac{1}{2}} \|\delta^{\frac{1}{2}} u_{xx}\|_2
 \| (  \varepsilon\eta)^{\frac{1}{2}}  v^2 v_{xxxx} \|_2  .
\end{align*}
 As a result, due to (\ref{nnn-1-001}), we find that
\begin{equation}\label{dd-1}
\delta  \iint \limits_{Q_T} { \bigl( \tfrac{v_x}{v} \bigr)_{xx}u_{xx} \,dx dt } \leqslant
C\,\delta^{\frac{1}{2}} \eta^{-\frac{1}{2}} \varepsilon^{- \frac{13}{2}} .
\end{equation}
Using (\ref{intr-tr}) and the estimates
$$
\int \limits_{\Omega} {v  v^2_x v^2_{xx} \,dx} \leqslant \tfrac{1}{3}
 \Bigl( \int \limits_{\Omega} { v^3 v^2_{xxx}  \,dx} \Bigr)^{\frac{1}{2}}
  \Bigl( \int \limits_{\Omega} {  \tfrac{v^6_{x}}{v}  \,dx} \Bigr)^{\frac{1}{2}}, \ \
\| v_x \|_{6} \leqslant C \| v_{xxx} \|^{\frac{1}{6}}_{ 2}  \| v_{x} \|^{\frac{5}{6}}_{ 2},
$$
we find that
\begin{align*}
    \int \limits_{\Omega} & {  \bigl( \tfrac{v_x}{v} \bigr)_{x} v (v^3 v_{xxx})_x  \,dx} = -
  \int \limits_{\Omega} {  \bigl( \tfrac{v_x}{v} \bigr)_{xx}  v^4 v_{xxx}  \,dx} -
  \int \limits_{\Omega} {  \bigl( \tfrac{v_x}{v} \bigr)_{x} v^3 v_x v_{xxx}  \,dx} \\
  & =  -   \int \limits_{\Omega} {    v^3 v^2_{xxx}  \,dx} + 2  \int \limits_{\Omega} {v^2 v_x v_{xx} v_{xxx}  \,dx} -
 \int \limits_{\Omega} { v v_x^3 v_{xxx}  \,dx} \\
 & \leqslant 
 -   \int \limits_{\Omega} {    v^3 v^2_{xxx}  \,dx}  + \tfrac{2}{\sqrt{3}}
 \Bigl( \int \limits_{\Omega} { v^3 v^2_{xxx}  \,dx} \Bigr)^{\frac{3}{4}}
  \Bigl( \int \limits_{\Omega} {  \tfrac{v^6_{x}}{v}  \,dx} \Bigr)^{\frac{1}{4}} \\
  &\quad + \Bigl( \int \limits_{\Omega} { v^3 v^2_{xxx}  \,dx} \Bigr)^{\frac{1}{2}}
  \Bigl( \int \limits_{\Omega} {  \tfrac{v^6_{x}}{v}  \,dx} \Bigr)^{\frac{1}{2}} \leqslant
 - \tfrac{1}{2}  \int \limits_{\Omega} {    v^3 v^2_{xxx}  \,dx}  + C \| v^{-1}\|_{\infty} \| v_x \|^6_{6} \\
 &\leqslant - \tfrac{1}{2}  \int \limits_{\Omega} {    v^3 v^2_{xxx}  \,dx}  + C \| v^{-1}\|_{\infty} \| v_x \|^5_{2}
  \| v_{xxx} \|_{2} \\
  & \leqslant 
   - \tfrac{1}{2}  \int \limits_{\Omega} {    v^3 v^2_{xxx}  \,dx}  +
  C \| v^{-1}\|^{\frac{5}{2}}_{\infty} \| v_x \|^5_{2} \|    v^{\frac{3}{2}} v_{xxx} \|_2 \\
  & \leqslant  
   - \tfrac{1}{4}  \int \limits_{\Omega} {    v^3 v^2_{xxx}  \,dx}  +
  C \| v^{-1}\|^{5}_{\infty} \| v_{xx} \|^{10}_{2}  \leqslant
 - \tfrac{1}{4}  \int \limits_{\Omega} {    v^3 v^2_{xxx}  \,dx}  +
  C \, \varepsilon^{-10} .
 \end{align*}
As a result, due to (\ref{nnn-1-001}), we arrive at
\begin{equation}\label{dd-2}
\eta \iint \limits_{Q_T} { \bigl( \tfrac{v_x}{v} \bigr)_{x} v (v^3 v_{xxx})_x  \,dx dt } \leqslant
 - \tfrac{\eta}{4}  \iint \limits_{Q_T} {    v^3 v^2_{xxx}  \,dx dt} + C\,\eta \, \varepsilon^{- 10}  .
\end{equation}
Using the estimate
\begin{align*}
\int \limits_{\Omega} { v (v^3 v_{xxx})^2_x \,dx} & \leqslant C \int \limits_{\Omega} {  v^7 v^2_{xxxx}  \,dx}  +
 C \int \limits_{\Omega} {  v^5 v_x^2 v^2_{xxx}  \,dx} \\ 
 & \leqslant  
 C \|v\|^3_{\infty} \int \limits_{\Omega} {  v^4 v^2_{xxxx}  \,dx}  +
 C \|v\|^5_{\infty} \|v_x \|^2_{\infty} \int \limits_{\Omega} { v^2_{xxx}  \,dx} \\
 & \leqslant 
 C \bigl( \|v\|^3_{\infty} + \|v\|^5_{\infty} \|v_x \|^2_{\infty} \|v^{-1}\|^4_{\infty} \bigr)
  \int \limits_{\Omega} {  v^4 v^2_{xxxx}  \,dx} \\
  & \leqslant C\, \varepsilon^{-\frac{15}{2}} \int \limits_{\Omega} {  v^4 v^2_{xxxx}  \,dx},
\end{align*}
we deduce that
\begin{align*}
\eta \, c \iint \limits_{Q_T} { v u_{x} (v^3 v_{xxx})_x \,dx dt} & \leqslant 
\eta \, c \Bigl( \iint \limits_{Q_T} { v u^2_{x} \,dx dt} \Bigr)^{\frac{1}{2}}
\Bigl( \iint \limits_{Q_T} { v   (v^3 v_{xxx})^2_x \,dx dt} \Bigr)^{\frac{1}{2}} \\ 
&\leqslant   
\eta \, c \Bigl( \iint \limits_{Q_T} { v u^2_{x} \,dx dt} \Bigr)^{\frac{1}{2}}
\Bigl( C\, \varepsilon^{-\frac{15}{2}}  \iint \limits_{Q_T} { v^4 v^2_{xxxx}  \,dx dt} \Bigr)^{\frac{1}{2}} ,
\end{align*}
whence, due to (\ref{nnn-1-001}), we have
\begin{equation}\label{dd-3}
\eta \, c \iint \limits_{Q_T} { v u_{x} (v^3 v_{xxx})_x \,dx dt} \leqslant C\, \eta^{\frac{1}{2}} \varepsilon^{- \frac{17}{4}}.
\end{equation}
Using the estimate
\begin{align*}
\int \limits_{\Omega} {  \tfrac{ (v^4 v_{xxx})_x}{v} \,dx} & = \int \limits_{\Omega} {v^3 v_{xxxx}  \,dx} +
 4 \int \limits_{\Omega} {v^2 v_x v_{xxx}  \,dx} \\
 & \leqslant 
\| v^2 v_{xxxx}\|_2 ( \| v\|_2 + C\, \|v^{-1}\|^2_{\infty} \| v \|^2_{\infty} \|v_x\|_2 ) \\
& \leqslant 
C\, \varepsilon^{-\frac{7}{2}} \| v^2 v_{xxxx}\|_2 = C\, \eta^{-\frac{1}{2}} \varepsilon^{-4} \|(\varepsilon \eta )^{\frac{1}{2}} v^2 v_{xxxx}\|_2,
\end{align*}
we get
\begin{equation}\label{dd-4}
   \eta \tfrac{ c }{a} \iint \limits_{Q_T} {  \tfrac{ (v^4 v_{xxx})_x}{g(h)} \,dx dt} \leqslant
 C\, \eta^{ \frac{1}{2}} \varepsilon^{-4} .
\end{equation}
Integrating (\ref{n-4-001}) in time, taking into account (\ref{nnn-1-001}) and
(\ref{dd-1})--(\ref{dd-4}), we obtain
\begin{multline} \label{n-4-001-0}
\mathcal{S}_{1,\varepsilon}(u,v) +
\tfrac{2bc}{a} \iint \limits_{Q_T} { h f^3(h_x)h^2_{xx} \,dx dt}  +
\iint \limits_{Q_T} {  u^2   \,dx dt}  
+\tfrac{2bc}{a} \iint \limits_{Q_T} { h^{-1} f(h_x) \,dx dt} \\
+ \delta \iint \limits_{Q_T} {    u^2_{xx} \,dx dt }
+ \varepsilon a  \iint \limits_{Q_T}{ v^2_{xxx } \,dx dt } 
+ \tfrac{2\varepsilon a}{3} \iint \limits_{Q_T} {\tfrac{ v^2_{x }}{v^{4}} \,dx dt} + \tfrac{\eta c^2}{4a}  \iint \limits_{Q_T} {    v^3 v^2_{xxx}  \,dx dt} \\
\leqslant
 \mathcal{S}_{1,\varepsilon}(u_{\varepsilon\eta,0},v_{\varepsilon\eta,0})  
+ C(T) + C\,\delta^{\frac{1}{2}} \eta^{-\frac{1}{2}} \varepsilon^{- \frac{13}{2}}
+ C\,\eta \, \varepsilon^{- 10} + C\, \eta^{\frac{1}{2}} \varepsilon^{- \frac{17}{4}}
\end{multline}
for all $T \leqslant T_{\eta}$.
\end{proof}

\subsection{Compactness and limit processes}

\textbf{Passage to the limit $\delta \to 0$.}
Denote the corresponding solution to the approximate problem (\ref{aprr-1})--(\ref{aprr-4}) by
$(v_{\delta\eta\varepsilon}, u_{\delta\eta\varepsilon})$. Let $T \leqslant T_{\eta}$.
We study the compactness properties of the sequence
$(v_{\delta\eta\varepsilon}, u_{\delta\eta\varepsilon})$ by using the estimates derived in Lemma~\ref{apr-est-eda}.
From (\ref{apr-nn-1}) and (\ref{apr-nn-2}), we have that $\{ v_{\delta\eta\varepsilon} \}_{\delta > 0}$ is bounded in $L^{\infty}(0,T; H^2(\Omega))$ and
$\{ v_{\delta\eta\varepsilon,t} \}_{\delta > 0}$ is bounded in $L^2(Q_T)$. Therefore, using \cite[Lemma 7.19, p.~175]{vazquez2007porous},
we conclude that $\{ v_{\delta\eta\varepsilon} \}_{\delta > 0}$ is bounded in $C_{x,t}^{\frac{3}{2}, \frac{3}{8}}(\bar{Q}_T)$.
By the Arzela-Ascoli theorem, after possibly extracting a subsequence, we obtain that
$$
v_{\delta\eta\varepsilon} \mathop {\to} \limits_{\delta \to 0} v_{\eta\varepsilon} \text{ uniformly in }  \bar{Q}_T ,
$$
$$
v_{\delta\eta\varepsilon, t}  \mathop {\to} \limits_{\delta \to 0}  v_{\eta\varepsilon, t} \text{ weakly in } L^2(Q_T),
$$
whence
$$
v^{-1}_{\delta\eta\varepsilon} \mathop {\to} \limits_{\delta \to 0} v^{-1}_{\eta\varepsilon} \text{ uniformly in }  \bar{Q}_T .
$$
Also, by (\ref{apr-nn-2}) we obtain that
$$
v_{\delta\eta\varepsilon} \mathop {\to} \limits_{\delta \to 0} v_{\eta\varepsilon} \text{ weakly in }  L^2(0,T; H^{4}(\Omega)),
$$
$$
v_{\delta\eta\varepsilon} \mathop {\to} \limits_{\delta \to 0} v_{\eta\varepsilon} \text{ strongly in }  L^2(0,T; H^{3}(\Omega)).
$$
Next, we turn to compactness properties of $\{ u_{\delta\eta\varepsilon} \}_{\delta > 0}$.
By (\ref{r-rr-apr})--(\ref{r-rr-apr-4}) and the boundedness $v_{\delta\eta\varepsilon} $
away from zero,  we have that $\{ u_{\delta\eta\varepsilon } \}_{\delta > 0}$ is bounded in
$L^{\infty}(0,T; L^{2}(\Omega)) \cap L^2(0,T; H^{1}(\Omega))$. Moreover, $\{ ( v_{\delta\eta\varepsilon} u_{\delta\eta\varepsilon})_t \}_{\delta > 0}$  and
$\{ u_{\delta\eta\varepsilon,t} \}_{\delta > 0}$ are bounded in $L^2(0,T; H^{-2}(\Omega))$. Therefore, we have
$$
u_{\delta\eta\varepsilon} \mathop {\to} \limits_{\delta \to 0} u_{\eta\varepsilon} \text{ strongly in } L^2(Q_T),
$$
$$
u_{\delta\eta\varepsilon,x} \mathop {\to} \limits_{\delta \to 0} u_{\eta\varepsilon,x} \text{ weakly in } L^2(Q_T),
$$
$$
u_{\delta\eta\varepsilon, t}  \mathop {\to} \limits_{\delta \to 0}  u_{\eta\varepsilon, t} \ *-\text{ weakly in } L^2(0,T; H^{-2}(\Omega)),
$$
$$
v_{\delta\eta\varepsilon} u_{\delta\eta\varepsilon} \mathop {\to} \limits_{\delta \to 0}
v_{ \eta\varepsilon} u_{ \eta\varepsilon} \text{ strongly in } L^2(Q_T),
$$
$$
( v_{\delta\eta\varepsilon} u_{\delta\eta\varepsilon})_t \mathop {\to} \limits_{\delta \to 0}  ( v_{ \eta\varepsilon} u_{ \eta\varepsilon})_t \ *-\text{ weakly in } L^2(0,T; H^{-2}(\Omega)).
$$
Moreover, from (\ref{apr-nn-2}) we obtain
$$
\delta \Bigl | \iint \limits_{Q_T} {u_{xx} \psi_{xx}  \,dx dt} \Bigr | \leqslant \delta^{\frac{1}{2}}
\| \delta^{\frac{1}{2}} u_{xx}  \|_{L^2(Q_T)} \| \psi   \|_{L^2(0,T;H^2(\Omega))} \leqslant C \, \delta^{\frac{1}{2}}.
$$
Based on the convergence results obtained, we can take the limit as $\delta \to 0$ in (\ref{ident-1}) and (\ref{ident-2}).

\textbf{Passage to the limit $\eta \to 0$.}
Since $T_{\eta} \to +\infty $ as $\eta \to 0$, we can extend the results to any $T > 0$.
Now, we consider the compactness properties of the sequence $(v_{\eta\varepsilon}, u_{\eta\varepsilon})$
by using the estimates derived in Lemma~\ref{apr-est-eda}.
Due to (\ref{apr-nn-1}) and (\ref{apr-nn-2}), we have that $\{ v_{\eta\varepsilon} \}_{\eta > 0}$ is bounded in $L^{\infty}(0,T; H^2(\Omega))$ and $\{ v_{\eta\varepsilon,t} \}_{\eta > 0}$ is bounded in $L^2(0,T; H^{-1}(\Omega))$. Therefore, similar to
\cite[Lemma 2.1, p.~183]{B8}, we arrive at the conclusion that $\{ v_{\eta\varepsilon} \}_{\eta > 0}$ is bounded in
$C_{x,t}^{\frac{3}{2}, \frac{1}{4}}(\bar{Q}_T)$. By the Arzela-Ascoli theorem, after possibly extracting a subsequence,
we obtain that
$$
v_{\eta\varepsilon} \mathop {\to} \limits_{\eta \to 0} v_{\varepsilon} \text{ uniformly in }  \bar{Q}_T ,
$$
$$
v_{ \eta\varepsilon, t}  \mathop {\to} \limits_{\eta \to 0}  v_{ \varepsilon, t} \  *-\text{ weakly in } L^2(0,T; H^{-1}(\Omega)),
$$
whence
$$
v^{-1}_{\eta\varepsilon} \mathop {\to} \limits_{\eta \to 0} v^{-1}_{\varepsilon} \text{ uniformly in }  \bar{Q}_T .
$$
Also, by (\ref{apr-nn-2}) we obtain that
$$
v_{\eta\varepsilon} \mathop {\to} \limits_{\eta \to 0} v_{\varepsilon} \text{ weakly in }  L^2(0,T; H^{3}(\Omega)),
$$
$$
v_{\eta\varepsilon} \mathop {\to} \limits_{\eta \to 0} v_{\varepsilon} \text{ strongly in }  L^2(0,T; H^{2}(\Omega)).
$$
Next, we turn to compactness properties of $\{ u_{\eta\varepsilon} \}_{\eta > 0}$.
Using (\ref{r-rr-apr})--(\ref{r-rr-apr-4}) and the boundedness $v_{ \eta\varepsilon} $
away from zero,  we have that $\{ u_{ \eta\varepsilon } \}_{\eta > 0}$ is bounded in
$L^{\infty}(0,T; L^{2}(\Omega)) \cap L^2(0,T; H^{1}(\Omega))$. Moreover, 
$\{  v_{ \eta\varepsilon }  u^2_{ \eta\varepsilon } \}_{\eta > 0}$ is bounded in $L^p(Q_T)$ for $p \in (1,3)$, and $\{ ( v_{ \eta\varepsilon} u_{ \eta\varepsilon})_t \}_{\eta > 0}$  and $\{ u_{ \eta\varepsilon,t} \}_{\varepsilon > 0}$ are bounded in $L^2(0,T; H^{-2}(\Omega))$.
Therefore, we have
$$
u_{ \eta\varepsilon} \mathop {\to} \limits_{\eta \to 0} u_{ \varepsilon} \text{ strongly in } L^2(Q_T),
$$
$$
u_{ \eta\varepsilon,x} \mathop {\to} \limits_{\eta \to 0} u_{ \varepsilon,x} \text{ weakly in } L^2(Q_T),
$$
$$
u_{ \eta\varepsilon, t}  \mathop {\to} \limits_{\eta \to 0}  u_{ \varepsilon, t} \ *-\text{ weakly in } L^2(0,T; H^{-2}(\Omega)),
$$
$$
v_{ \eta\varepsilon} u_{ \eta\varepsilon} \mathop {\to} \limits_{\eta \to 0}
v_{  \varepsilon} u_{  \varepsilon} \text{ strongly in } L^2(Q_T),
$$
$$
v_{ \eta\varepsilon} u^2_{ \eta\varepsilon} \mathop {\to} \limits_{\eta \to 0}
v_{  \varepsilon} u^2_{\varepsilon} \text{ strongly in } L^2(Q_T),
$$
$$
( v_{ \eta\varepsilon} u_{ \eta\varepsilon})_t \mathop {\to} \limits_{\eta \to 0}  ( v_{ \varepsilon} u_{ \varepsilon})_t \ *-\text{ weakly in } L^2(0,T; H^{-2}(\Omega)).
$$
Moreover, by (\ref{r-rr-apr}) and (\ref{apr-nn-2}), we arrive at
\begin{align*}
\eta \iint \limits_{Q_T} {  u\, v^4 v_{xxx}  \psi_x \,dx dt} &\leqslant \eta
\int \limits_0^T { \| \sqrt{v} u \|_{2} \| v \|^{\frac{7}{2}}_{\infty} \| v_{xxx}\|_2 \| \psi_x \|_{\infty} \,dt }   \\
& \leqslant  C\,\eta \varepsilon^{- \frac{1}{2}} \| \varepsilon^{\frac{1}{2}}  v_{xxx}\|_{L^2(Q_T)}  \| \psi   \|_{L^2(0,T;H^2(\Omega))}
 \leqslant C\,\eta \varepsilon^{- \frac{1}{2}} ,
\end{align*}
\begin{align*}
\eta \iint \limits_{Q_T} {  v^4 v_{xxx}  \phi_x \,dx dt} &\leqslant \eta
\varepsilon^{-\frac{1}{2}} \| v \|_{L^\infty(Q_T)} \| \varepsilon^{\frac{1}{2}}  v_{xxx}\|_{L^2(Q_T)} \| \psi_x \|_{L^2(Q_T)}   \\
& \leqslant  
 C\,\eta \varepsilon^{- \frac{1}{2}}   \| \psi   \|_{L^2(0,T;H^1(\Omega))}  \leqslant C\,\eta \varepsilon^{- \frac{1}{2}} .
\end{align*}
The obtained convergence results enable us to take the limit as $\eta \to 0$ in (\ref{ident-1}) and (\ref{ident-2}) with $\delta = 0$.

\textbf{Passage to the limit $\varepsilon \to 0$.}
Next,  we study the compactness properties of the sequence $(v_{\varepsilon}, u_{\varepsilon})$
by using the estimates derived in Lemma~\ref{apr-est-eda}. Taking into account
$$
(\sqrt{v})_t = - a  ( \sqrt{v} u)_x  +  \tfrac{a}{2} \sqrt{v} u_x ,
$$
by (\ref{r-rr-apr}) and (\ref{r-rr-apr-3}), we deduce that $\{  (\sqrt{v_{\varepsilon }})_t \}_{\varepsilon > 0}$ is uniformly
bounded in $L^{2}(0,T; H^{-1}(\Omega))$, and $ \{ \sqrt{v_{\varepsilon}} \}_{\varepsilon > 0}$ is uniformly
bounded in $L^{\infty}(0,T; H^1(\Omega))$. Therefore, based on the lemma of compactness embedding from  \cite[Corollary 4, p.~85]{simon1986compact}, we obtain that
\begin{equation}\label{boun-1}
\sqrt{v_{ \varepsilon} }\mathop {\to} \limits_{\varepsilon \to 0} \sqrt{v}  \text{ uniformly in }  \bar{Q}_T,
\end{equation}
and it follows that
\begin{equation}\label{boun-1-00}
 v_{ \varepsilon}  \mathop {\to} \limits_{\varepsilon \to 0} v  \text{ uniformly in }  \bar{Q}_T.
\end{equation}
Also, by (\ref{r-rr-apr}) and (\ref{boun-1}),  $\{ u_{\varepsilon} v_{\varepsilon} \}_{\varepsilon > 0}$ is uniformly
bounded in $L^{2}(Q_T)$. Therefore, we find that $\{  v_{\varepsilon, t} \}_{\varepsilon > 0}$ is uniformly
bounded in $L^{2}(0,T; H^{-1}(\Omega))$, which implies that
$$
v_{\varepsilon, t}  \mathop {\to} \limits_{\varepsilon \to 0}  v_{t} \  *-\text{ weakly in } L^2(0,T; H^{-1}(\Omega)).
$$
From the boundedness of  $\{ h_{\varepsilon} \Phi(h_{\varepsilon,x}) \}_{\varepsilon > 0}$
in $L^{\infty}(0,T; L^1(\Omega))$, we deduce that
\begin{equation}\label{bound-2}
\{ v_\varepsilon \}_{\varepsilon > 0} \text{ is uniformly bounded in }
L^{\infty}(0,T; W_1^1(\Omega)).
\end{equation}
Therefore, from (\ref{boun-1-00}), it follows that
$$
v_{\varepsilon }, \, h_{\varepsilon }  \mathop {\to} \limits_{\varepsilon \to 0}  v, \, h  \  *-\text{ weakly in } L^{\infty}(0,T;  W_1^1(\Omega)),
$$
and the set $  \{ |h_x(.,t)| = \infty \}  $ has Lebesgue measure zero for any $t > 0$.
By the boundedness of $\{ \log (v_{\varepsilon})\}_{\varepsilon > 0}$ in $L^{\infty}(0,T; L^1(\Omega))$ and
(\ref{boun-1-00}), the set $  \{ v(.,t) = 0\}  $ has Lebesgue measure zero
for any $t > 0$, and it follows that
$$
p(v_{\varepsilon})  \mathop {\to} \limits_{\varepsilon \to 0} p(v)
$$
holds for almost all $x$ and for any $t> 0$, where $p(z) = \frac{1}{2} z^{-2}$.

In the case of the laminar flow model with $g(h) = \frac{I(h)}{v}$,
from (\ref{r-rr-apr}) and (\ref{r-rr-apr-2}), due to (\ref{boun-1-00}), we obtain that
$\sqrt{v} \in L^\infty(0,T;H^1(\Omega))$ and $\int \limits_{\Omega} {G(h)\,dx} < + \infty $.
Next, we show that $v(x,t) > 0$ in $Q_T$ by contradiction. Assume that there exists a point $x_0 \in \bar{\Omega}$
such that $v(x_0,t) = 0$, then $v(x,t) \leqslant C|x -x_0|$. Taking into account  $ G(h) \sim \frac{C}{h -1}$ as $h \to 1$ (see Remark \ref{remark_G}),
we find that
$$
+ \infty > \int \limits_{\Omega} {G(h)\,dx}  \geqslant C \int \limits_{\Omega} {  \tfrac{dx}{|x-x_0|}} = + \infty.
$$
This contradiction proves that $v  > 0$.

By (\ref{boun-1-00}),  $\varepsilon  p(v_{\varepsilon}) \to 0$  uniformly on $\{ v > \nu \}$ as $\varepsilon \to 0$  for any
$\nu > 0$. Then
$$
\iint \limits_{\{ v > \nu \} } {\varepsilon p(v_\varepsilon) \psi_x \, dx dt} \to 0  \text{ as } \varepsilon \to 0
$$
for any $\nu > 0$. By  (\ref{apr-nn-1}) and (\ref{apr-nn-2}),  $\{ \varepsilon^{1/2} v^{-1}_\varepsilon \}_{\varepsilon > 0} $ is uniformly bounded in $L^\infty(0,T;L^2(\Omega)) \cap  L^2(0,T;H^1(\Omega)$.
Using \cite[Proposition 3.3, p.10]{dibenedetto1993degenerate},
we obtain that
$\{ \varepsilon^{1/2} v^{-1}_\varepsilon \} $ is uniformly bounded in
$L^4(0,T;L^\infty (\Omega)) $. From here, if $\varepsilon$ is sufficiently small, depending on $\nu$,
we deduce that
\begin{align*}
\Bigl |  \iint \limits_{\{ v \leqslant \nu \} } {\varepsilon p(v_\varepsilon) \psi_x \, dx dt}   \Bigr| & =
\Bigl | \tfrac{1}{2} \iint \limits_{\{ v \leqslant \nu \} } { (\varepsilon^{1/2}  v^{-1}_\varepsilon )^2 \psi_x \, dx dt}   \Bigr|      \\
& \leqslant \tfrac{1}{2}    \| \varepsilon^{1/2} v^{-1}_\varepsilon \|^2_{L^4(0,T;L^\infty (\Omega))}
( \mathop {\sup} \limits_{t \in [0,T] } | \{  v(.,t) \leqslant \nu\} | )^{1/2} \| \psi_x \|_{L^2(Q_T)} \\
& \leqslant 
C \, ( \mathop {\sup} \limits_{t \in [0,T] } | \{ v(.,t) \leqslant \nu \} | )^{1/2} .
\end{align*}
As a result, since $\nu > 0$ is arbitrary and $ | \{ v(.,t) = 0\} | = 0  $ for any $t > 0$, we take $\varepsilon \to 0$ and arrive at
$$
\iint \limits_{Q_T} {\varepsilon p(v_\varepsilon) \psi_x \, dx dt} \to 0  \text{ as } \varepsilon \to 0.
$$

Let us denote by $F(z) = \int \limits_{-\infty}^z {f^{\frac{3}{2}}(s)\, ds} $ and $F'(\pm \infty) = 0$.
By (\ref{r-rr-apr-3}), we find that $\{ F(h_{\varepsilon,x}) \}_{\varepsilon > 0}$ is uniformly bounded in $L^{2}(0,T; H^1(\Omega))$. Therefore, we have
$$
F(h_{\varepsilon,x})  \mathop {\to} \limits_{\varepsilon \to 0}  F(h_x)  \text{ weakly in } L^{2}(0,T; H^1(\Omega)),
$$
$$
F(h_{\varepsilon,x})  \mathop {\to} \limits_{\varepsilon \to 0}  F(h_x)  \text{ strongly in } L^{2}(Q_T)
\text{ and a.\,e. in } Q_T,
$$
and it follows that
\begin{equation}\label{rtt-1}
h_{\varepsilon,xx}   \mathop {\to} \limits_{\varepsilon \to 0}   h_{xx}  \text{ weakly in } L^{2}(\{|h_x| < K \}) ,
\end{equation}
\begin{equation}\label{rtt-2}
h_{\varepsilon,x}   \mathop {\to} \limits_{\varepsilon \to 0}   h_{x}  \text{ strongly in } L^{2}(\{|h_x| < K \})
\end{equation}
for any $K > 0$. In particular, in view of $  | \{ |h_x(.,t)| = \infty \}| =0 $ for any $t > 0$, we obtain that
\begin{equation}\label{rtt-3}
 h_{\varepsilon,x}   \mathop {\to} \limits_{\varepsilon \to 0}   h_x \text{ strongly in }L^2(Q_T)  \text{ and a.\,e. in } Q_T.
\end{equation}
Using the following estimates
$$
| \kappa_{\varepsilon} | = \bigl | h^{-1}_{\varepsilon} f(h_{\varepsilon,x}) -
\sqrt{ \tfrac{f^3(h_{\varepsilon,x})}{h_{\varepsilon} }} \sqrt{h_{\varepsilon} f^3(h_{\varepsilon,x})  }  h_{\varepsilon,xx} \bigr |
\leqslant 1 + \sqrt{h_{\varepsilon} f^3(h_{\varepsilon,x})  } |h_{\varepsilon,xx}| ,
$$
\begin{align*}
| \kappa_{\varepsilon} v_{\varepsilon,x} | &= 2 \bigl | h_{\varepsilon,x} f(h_{\varepsilon,x}) -
h_{\varepsilon,x} \sqrt{ h_{\varepsilon} f^3(h_{\varepsilon,x})} \sqrt{h_{\varepsilon} f^3(h_{\varepsilon,x})  }  h_{\varepsilon,xx} \bigr | \\
& \leqslant   
2 + 2\| \sqrt{h_{\varepsilon}} \|_{\infty} \sqrt{h_{\varepsilon} f^3(h_{\varepsilon,x})  } |h_{\varepsilon,xx}|  ,
\end{align*}
due to (\ref{boun-1-00}) and (\ref{r-rr-apr-3}), we have that $\{ \kappa_{\varepsilon} \}_{\varepsilon > 0}$ and
$\{ \kappa_{\varepsilon} v_{\varepsilon,x} \}_{\varepsilon > 0}$ are  uniformly bounded in $L^{2}(Q_T)$.
Therefore, from (\ref{rtt-1})-- (\ref{rtt-3}) and (\ref{boun-1-00}), for any $K > 0$, we have
$$
 \iint \limits_{\{ |h_x| < K\} } {  f^3(h_{\varepsilon,x}) h_{\varepsilon,xx}  v_\varepsilon \psi_x  \,dx dt}
\mathop {\to} \limits_{\varepsilon \to 0} \iint \limits_{\{ |h_x| < K\} } {  f^3(h_{ x}) h_{ xx}  v \, \psi_x  \,dx dt},
$$
$$
  \iint \limits_{\{ |h_x| <  K\} } { h_{\varepsilon,x}   h_{\varepsilon} f^3(h_{\varepsilon,x}) h_{\varepsilon,xx} \psi  \,dx dt}
\mathop {\to} \limits_{\varepsilon \to 0} \iint \limits_{\{ |h_x| <  K\} } { h_{x}   h  f^3(h_{ x}) h_{ xx} \psi  \,dx dt}  ,
$$
$$
 \iint \limits_{Q_T} {  h^{-1}_{\varepsilon} f(h_{\varepsilon,x})  v_\varepsilon \psi_x  \,dx dt}
\mathop {\to} \limits_{\varepsilon \to 0} \iint \limits_{Q_T } { h^{-1}  f(h_{ x})  v \, \psi_x  \,dx dt},
$$
$$
  \iint \limits_{Q_T } { h_{\varepsilon,x} f(h_{\varepsilon,x})  \psi  \,dx dt}
\mathop {\to} \limits_{\varepsilon \to 0} \iint \limits_{Q_T } { h_{ x} f(h_{ x})   \psi  \,dx dt}  .
$$
On the other hand, if $\varepsilon$ is sufficiently small (depending on $K$), then by (\ref{r-rr-apr-3}) and using the form $f(z) = (1+z^2)^{-1/2}$ and the inequality $z^2f^3(z)\leqslant f(z)$, we obtain
\begin{align*}
\Bigl | \iint \limits_{\{ |h_x| \geqslant K\} } {  f^3(h_{\varepsilon,x}) h_{\varepsilon,xx}  v_\varepsilon \psi_x  \,dx dt}
\Bigr | 
& \leqslant    
\tfrac{C}{(1+K^2)^{3/4}}\Bigl(  \iint \limits_{Q_T} {  h_{\varepsilon} f^3(h_{\varepsilon,x})  h^2_{\varepsilon,xx}   \,dx dt}  \Bigr)^{\frac{1}{2}} \\
& \leqslant \tfrac{C}{(1+K^2)^{3/4}},
\end{align*}
 \begin{align*}
 \Bigl | \iint \limits_{\{ |h_x| \geqslant K\} } { h_{\varepsilon,x}   h_{\varepsilon} f^3(h_{\varepsilon,x}) h_{\varepsilon,xx} \psi  \,dx dt}
\Bigr | & \leqslant  
 \tfrac{C}{(1+K^2)^{1/4}}\Bigl(  \iint \limits_{Q_T} {  h_{\varepsilon} f^3(h_{\varepsilon,x})  h^2_{\varepsilon,xx}   \,dx dt}  \Bigr)^{\frac{1}{2}} \\
 & \leqslant  \tfrac{C}{(1+K^2)^{1/4}}.
 \end{align*}
As a result, since $K > 0$ is arbitrary, we take $\varepsilon \to 0$ and deduce that
$$
\iint \limits_{Q_T} { \kappa_\varepsilon  v_\varepsilon \psi_x  \,dx dt}
\mathop {\to} \limits_{\varepsilon \to 0} \iint \limits_{Q_T} {
( h^{-1}  f(h_{x}) -  \chi_{\{  |h_x| < \infty \}} f^3(h_{x}) h_{xx} ) v \psi_x  \,dx dt},
$$
$$
\iint \limits_{Q_T} { \kappa_\varepsilon  v_{\varepsilon,x}  \psi  \,dx dt} \mathop {\to} \limits_{\varepsilon \to 0}
\iint \limits_{Q_T} {
( 2 \Phi' (h_{x}) -  \chi_{\{  |h_x| < \infty \}} f^3(h_{x})v_x h_{xx} ) \psi   \,dx dt} .
$$
By (\ref{apr-nn-2}), (\ref{bound-2}), and
$$
\|v_{xx} \|_2 \leqslant C \|v_{xxx} \|^{\frac{3}{5}}_2 \|v_{x} \|^{\frac{2}{5}}_1, \  \
\|v_{x} \|_2 \leqslant C \|v_{xxx} \|^{\frac{1}{5}}_2 \|v_{x} \|^{\frac{4}{5}}_1,
$$
we have
\begin{align*}
\varepsilon & \Bigl |  \iint \limits_{Q_T} {   v_{xxx} ( v_{xx} \psi + 2 v_x \psi_x + v \psi_{xx} )   \,dx dt}
\Bigr | \leqslant   \varepsilon \int \limits_{0}^T {   \|v_{xxx} \|_2 \|v_{xx} \|_2 \|\psi \|_{\infty} \,dt}  \\
& + 2\varepsilon \int \limits_{0}^T { \|v_{xxx} \|_2 \|v_{x} \|_2 \|\psi_x \|_{\infty}  \,dt} +
\varepsilon \int \limits_{0}^T {  \|v_{xxx} \|_2 \|\psi_{xx} \|_2 \|v \|_{L^\infty(Q_T)} \,dt} \\
& \leqslant 
C\,\varepsilon^{\frac{1}{5}}  \| \varepsilon^{\frac{1}{2}}  v_{xxx}\|^{\frac{8}{5}}_{L^2(Q_T)}
\| v_x \|^{\frac{2}{5}}_{L^{\infty}(0,T;L^{1}(\Omega))} \| \psi \|_{L^5(0,T;L^{\infty}(\Omega))} \\
& \quad + 
C\,\varepsilon^{\frac{2}{5}}  \| \varepsilon^{\frac{1}{2}}  v_{xxx}\|^{\frac{6}{5}}_{L^2(Q_T)}
\| v_x \|^{\frac{4}{5}}_{L^{\infty}(0,T;L^{1}(\Omega))} \| \psi_x \|_{L^\frac{5}{2}(0,T;L^{\infty}(\Omega))} \\
& \quad + \varepsilon^{\frac{1}{2}}  \| \varepsilon^{\frac{1}{2}}  v_{xxx}\|_{L^2(Q_T)}  \| \psi   \|_{L^2(0,T;H^2(\Omega))}
\leqslant C\, \varepsilon^{\frac{1}{5}} .
\end{align*}
By (\ref{r-rr-apr-3}), $\{ u_{\varepsilon} \}_{\varepsilon > 0}$ is uniformly
bounded in $L^{2}(Q_T)$, and $\{  u_{\varepsilon,x} \}_{\varepsilon > 0}$
is uniformly bounded in $L^{2}(\{ v > \mu\})$ for any $\mu > 0$. Therefore, we have
\begin{equation}\label{boun-1-01}
u_{ \varepsilon}  \mathop {\to} \limits_{\varepsilon \to 0} u \text{ weakly in } L^2(Q_T) ,
\end{equation}
\begin{equation}\label{rtt-4}
u_{ \varepsilon,x}  \mathop {\to} \limits_{\varepsilon \to 0} u_x \text{ weakly in } L^2(\{ v > \mu\}),
\end{equation}
\begin{equation}\label{rtt-5}
u_{ \varepsilon }  \mathop {\to} \limits_{\varepsilon \to 0} u  \text{ strongly in } L^2(\{ v > \mu\})
\text{ and a.\,e. in } \{ v > \mu\} .
\end{equation}
By (\ref{boun-1-00}) and (\ref{rtt-4}), we obtain
$$
\iint \limits_{\{ v > \mu\} } { v_{\varepsilon }u_{\varepsilon,x}  \psi_x  \,dx dt} \mathop {\to} \limits_{\varepsilon \to 0}
\iint \limits_{ \{ v > \mu\} } { v u_{x} \psi_x   \,dx dt} .
$$
On the other hand, if $\varepsilon$ is sufficiently small (depending on $\mu$), then by (\ref{r-rr-apr-3}), we arrive at
$$
\Bigl | \iint \limits_{\{ v \geqslant \mu \} } {  v_{\varepsilon }u_{\varepsilon,x}  \psi_x  \,dx dt}
\Bigr | \leqslant
C\, \mu^{\frac{1}{2}} \Bigl(  \iint \limits_{Q_T} {  v_{\varepsilon }u^2_{\varepsilon,x}  \,dx dt}  \Bigr)^{\frac{1}{2}} \leqslant C\, \mu^{\frac{1}{2}} .
$$
From here, since $\mu > 0$ is arbitrary, we take $\varepsilon \to 0$ and deduce that
$$
\iint \limits_{Q_T} {v_{\varepsilon }u_{\varepsilon,x}  \psi_x \,dx dt}
\mathop {\to} \limits_{\varepsilon \to 0} \iint \limits_{Q_T} {
  \chi_{\{  v > 0 \}} v \, u_{x}  \psi_x  \,dx dt}.
$$
Similarly, using (\ref{boun-1-00}), (\ref{rtt-5})  and (\ref{r-rr-apr-3}), we arrive at
$$
\iint \limits_{Q_T} { v_{\varepsilon }u^2_\varepsilon  \psi_x  \,dx dt} \mathop {\to} \limits_{\varepsilon \to 0}
\iint \limits_{Q_T} {\chi_{\{  v > 0 \}} v u^2 \psi_x   \,dx dt} .
$$
By (\ref{boun-1-00}) and (\ref{boun-1-01}), we get
$$
u_{\varepsilon} v_{\varepsilon}  \mathop {\to} \limits_{\varepsilon \to 0}   u v \text{ weakly in } L^2(Q_T),
$$
and therefore we have
$$
\iint \limits_{Q_T} { u_{\varepsilon} v_{\varepsilon}  \psi_x  \,dx dt} \mathop {\to} \limits_{\varepsilon \to 0}
\iint \limits_{Q_T} { u v   \psi_x   \,dx dt} .
$$
Using the obtained convergence results, we pass to the limit as $\varepsilon \to 0$ in (\ref{ident-1}) and (\ref{ident-2}),
with $\delta = \eta = 0$. As a result, we obtain a weak solution $(h,u)$ in the sense of Definition~\ref{Def-weak}.

\section{Travelling wave solutions}\label{sec:travelling}

Next, we focus on the travelling wave solutions to the control-volume model.
Specifically, we look for a solution to \eqref{r-1-0}-- \eqref{r-2-0} in the form:
$$
u(x,t) = U(\xi), \ \ \ v(x,t) = V(\xi) = H^2(\xi)-1, \text{ where } \xi = x - s\,t,
$$
where $s$ is the propagation speed.
Substituting the ansatz into \eqref{r-1-0}-- \eqref{r-2-0}, we obtain the system of travelling wave ODEs for $(U(\xi), V(\xi))$ for $0\leqslant \xi \leqslant L$,
\begin{equation}\label{t-1}
- s\,U' + a U \,U'  + b\, \kappa' = c\,\frac{(V \,U')'}{V} + 1 - \frac{U}{g(H)} ,
\end{equation}
\begin{equation}\label{t-2}
-s\,V' + a( U \, V )'   = 0
\end{equation}
subject to the $L$-periodic boundary conditions. We also impose the following mass constraint
$$
\int \limits_0^L {H^2(\xi)\, d\xi} = \mathcal{M} > 0,
$$
where $\mathcal{M}$ is related to the mass $M$ defined in \eqref{mass} by $\mathcal{M} = M + L$.
We will look for $L$ (depending on $M$) such that $V(\xi) > 0$ for $\xi \in (0,L)$ and $U$ are continuous
functions. 

To study the structure of travelling wave solutions, first we consider a general travelling wave solution that satisfies the periodic boundary condition
\begin{equation}\label{t-3-0}
V(0) = V(L) > 0 .
\end{equation}
From (\ref{t-2}) it follows that
\begin{equation}\label{tt}
- V (s -a\,U) = C_0,
\end{equation}
which implies that
\begin{equation}\label{t-4-0}
U = U_c + \frac{C_0}{a\,V}  \ \ \forall\,C_0 \in \mathbb{R}^1,  \text{ where } U_c  := \frac{s}{a}.
\end{equation}

\begin{lemma}\label{lem-tw-2}
There exist $L > 0$, $s $ and $C_0 \neq 0$ such that the problem (\ref{t-1})--(\ref{t-2})  has at least one periodic
(non-trivial) solution $(H,U)$ satisfying
$$
H(0) = H(L) > 1, \ H'(0) = H'(L) = 0.
$$
\end{lemma}

\begin{proof}[Proof of Lemma~\ref{lem-tw-2}]
From (\ref{t-1}) it follows that
$$
\left[b\, \kappa + \tfrac{a}{2}(U - U_c)^2  \right]' = c \tfrac{(V U' )'}{V} +  1 - \tfrac{U}{g(H)},
$$
whence, due to (\ref{t-4-0}), we obtain
\begin{equation}\label{tt-00}
\left[b\, \kappa + \tfrac{C_0^2}{2a}V^{-2}   \right]' = - \tfrac{c\,C_0}{a} \tfrac{1}{V} \left( \tfrac{V'}{V}\right)'   +  1 - \tfrac{U_c}{g(H)}
- \tfrac{C_0}{a \,Vg(H)}.
\end{equation}
Let $C_0 \neq 0$. Then we have
\begin{equation}\label{t-7-b}
\left[H f(H') - \tfrac{C_0^2}{4ab}V^{-1} \right]' = (C_1 - G(\xi)) H H' \ \ \forall\, C_1 \in \mathbb{R}^1,
\end{equation}
where
$$
G(\xi) :=  b^{-1} \int \limits_{\xi_0}^{\xi} { \Bigl(  \tfrac{U_c}{g(H(y))} + \tfrac{C_0}{a \,V(y)g(H(y))} - 1  + \tfrac{c\,C_0}{a} \tfrac{1}{V} \left( \tfrac{V'}{V}\right)' \Bigr) \, dy  }.
$$
By imposing the periodicity $G(0) = G(L) $, we obtain
\begin{equation}\label{rrm-1}
U_c \int \limits_0^{L} { \tfrac{dy}{g(H(y))}}  + \tfrac{C_0}{a} \int \limits_0^{L} { \bigl( \tfrac{1}{V(y)g(H(y))} + c\, \tfrac{V'^2(y)}{V^{3}(y)} \bigr)dy } = L.
\end{equation}
Integrating (\ref{t-7-b}), we deduce that
\begin{equation}\label{t-8}
f(H') = A(\xi) H + B(\xi) H^{-1} + \tfrac{C_0^2}{4ab}H^{-1}V^{-1},
\end{equation}
where
$$
A(\xi) = C_1 - \tfrac{1}{2}G(\xi), \ \  B(\xi) = C_2  + \tfrac{1}{2} \int \limits_{\xi_0}^{\xi} {G'(y) H^2(y) \, dy }
$$
and $C_i$ for $i=1,2$ are arbitrary constants  such that $A(0) = A(L)$ and $B(0) = B(L)$. From $B(0) = B(L)$, it follows that
\begin{equation}\label{rrm-2}
U_c \int \limits_0^{L} { \tfrac{H^2(y) dy}{g(H(y))}}  + \tfrac{C_0}{a} \int \limits_0^{L} { \bigl( \tfrac{H^2(y)}{V(y)g(H(y))} +
c\, \tfrac{V'^2(y)}{V^{3}(y)}\bigr) dy } = \mathcal{M}.
\end{equation}
Therefore, solving system (\ref{rrm-1}), (\ref{rrm-2}) for $U_c$ and $C_0$, we find that
$$
U_c = \frac{1}{\newDelta} \Bigl( L \int \limits_0^{L} {\tfrac{H^2(y)dy}{V(y)g(H(y))}} -
\mathcal{M} \int \limits_0^{L} {  \tfrac{dy}{V(y)g(H(y))}}   + c(L-\mathcal{M}) \int \limits_0^{L} {
\tfrac{V'^2(y)}{V^{3}(y)}\, dy}\Bigr),
$$
$$
C_0 = -{\frac{a}{{\newDelta}}} \Bigl( L \int \limits_0^{L} { \tfrac{H^2(y) dy}{g(H(y))}} - \mathcal{M} \int \limits_0^{L} { \tfrac{dy}{g(H(y))}} \Bigr),
$$
where
\begin{align*}
\newDelta = & \Bigl(   \int \limits_0^{L} { \tfrac{ dy}{g(H(y))}}  \Bigr) \Bigl(   \int \limits_0^{L} { \tfrac{H^2(y)dy}{V(y)g(H(y))}}  \Bigr)
- \Bigl(   \int \limits_0^{L} { \tfrac{H^2(y) dy}{g(H(y))}}  \Bigr)\Bigl(   \int \limits_0^{L} {   \tfrac{dy}{V(y)g(H(y))}}  \Bigr) \\
&- c \Bigl(   \int \limits_0^{L} { \tfrac{V(y)dy}{ g(H(y))} }  \Bigr)  \Bigl(   \int \limits_0^{L} {
\tfrac{V'^2(y)}{V^{3}(y)}\, dy}  \Bigr) .
\end{align*}
As a result, by (\ref{t-8}) we arrive at
$$
H'^2(\xi) = \frac{1 - \left[A(\xi) H + B(\xi) H^{-1} + \frac{C_0^2}{4ab}H^{-1}V^{-1}\right]^2 }{ \left[A(\xi) H + B(\xi) H^{-1} + \frac{C_0^2}{4ab}H^{-1}V^{-1}\right]^2}.
$$
Furthermore, if we select $C_i$ for $i=1,2$ such that $A(0) = B(0) = 0$, then we have 
$H'(0) = H'(L) = 0$ provided that $H(0) = H(L) > 1$ satisfy the following
equation
$$
H(0) (H^2(0) -1)  =   \tfrac{C_0^2}{4ab},
$$
which has one solution if $C_0 \neq 0$.

%
\end{proof}

Next, we consider a special case when the film profile
touches down to zero at the boundary. That is, we assume that
\begin{equation}\label{t-3}
V(0) = V(L) = 0, \quad \text{or, equivalently, } H(0) = H(L) = 1.
\end{equation}
From \eqref{tt} and (\ref{t-3}), we obtain $C_0 =0$, and $U(\xi)$ becomes a trivial solution
\begin{equation}\label{t-4}
U \equiv U_c := \frac{s}{a}.
\end{equation}

\begin{lemma}\label{lem-tw}
There exist $L > 0$ and $s$ such that the problem (\ref{t-1})--(\ref{t-2})   has at least one periodic solution $(H,U)$ such that
$$
H(0) = H(L) = 1, \ H'(0) = H'(L) ,
$$
where the average fluid film radius
\begin{equation}\label{asss}
\bar{\mathcal{M}} := \frac{\mathcal{M}}{L} =  1 + \tfrac{\int \limits_0^{L} { \tfrac{H^2(y)-1}{g(H(y))}   dy } }{ \int \limits_0^{L} { \tfrac{dy}{g(H(y))}} }.
\end{equation}
\end{lemma}

\begin{remark}
If $g(H) = H^2 - 1$ (the plug flow case), then (\ref{asss}) implies that
$$
\bar{\mathcal{M}} = 1+ \frac{s}{a}   .
$$
\end{remark}

\begin{proof}[Proof of Lemma~\ref{lem-tw}]
Since $U$ is a trivial solution satisfying (\ref{t-4}), the ODE (\ref{t-1}) reduces to
\begin{equation}
b\, \kappa' =  1 - \frac{U_c}{g(H)}.
\nonumber
\end{equation}
Using the relation
\begin{equation}
\kappa =  f(H') H^{-1} - f^3(H')H'' = \frac{( H f(H'))'}{H H'},
\nonumber
\end{equation}
we obtain
\begin{equation}
\label{tr-00}
 b \left[ \frac{( H f(H'))'}{H H'} \right]' = 1 - \frac{U_c}{g(H)}.
\end{equation}
Integrating \eqref{tr-00} once, we have
\begin{equation}\label{t-5}
\frac{( H f(H'))'}{H H'}  = F(\xi) + C_1,
\end{equation}
where
$$
F(\xi)  := b^{-1} \int \limits_{\xi_0}^{\xi} { \bigl(1 - \tfrac{U_c}{g(H(y))}   \bigr) dy }.
$$
By periodicity, we find that $F(0) = F(L)$ which implies
$$
U_c =  \Bigl( \frac{1}{L}\int \limits_0^{L} { \frac{dy}{g(H(y))}   } \Bigr)^{-1}.
$$
From (\ref{t-5}), we deduce that
\begin{equation}\label{t-7}
f(H') = A(\xi) H + B(\xi)  H^{-1},
\end{equation}
where 
$$
A(\xi) = C_1 + \tfrac{1}{2}F(\xi), \  B(\xi) = C_2 -  \tfrac{1}{2} \int \limits_{\xi_0}^{\xi} { H^2(y) F'(y) \, dy },
$$
and $C_i$ for $i=1,2$ are arbitrary constants. We note that a necessary condition for the existence of real-valued solutions to the ODE \eqref{t-7} is $0 \leqslant A(\xi) H + B(\xi)  H^{-1}\leqslant 1$.
By periodicity, we find that
$$
\int \limits_{0}^{L} { \bigl(1 - \tfrac{U_c}{g(H(y))}   \bigr) dy } =
\int \limits_{0}^{L} { H^2(y)  \bigl(1 - \tfrac{U_c}{g(H(y))}  \bigr) \, dy },
$$
whence we get
$$
\mathcal{M} = L + U_c \int \limits_0^{L} { \tfrac{H^2(y)-1}{g(H(y))}   dy } = L + \Bigl( \tfrac{1}{L}\int \limits_0^{L} { \tfrac{dy}{g(H(y))}  } \Bigr)^{-1}
\int \limits_0^{L} { \tfrac{H^2(y)-1}{g(H(y))}   dy }  .
$$
Hence, $ \mathcal{\bar{M}} = \mathcal{M}$ satisfies
$$
 \mathcal{\bar{M}}  =  1 + \tfrac{\int \limits_0^{L} { \tfrac{H^2(y)-1}{g(H(y))}   dy } }{ \int \limits_0^{L} { \tfrac{dy}{g(H(y))}} }.
$$
Moreover, by (\ref{t-7}) we arrive at
$$
[H']^2 = \frac{1 - [A(\xi) H + B(\xi) H^{-1}]^2 }{ [A(\xi) H + B(\xi) H^{-1}]^2}.
$$
Furthermore, if we select $C_i$ for $i=1,2$ such that $A(0) = B(0) = 0$, then we have $H'(0) = H'(L) = \infty$.

\end{proof}

In section \ref{sec:numerics}, we present numerical studies of travelling wave solutions discussed in Lemma \ref{lem-tw-2} that do not touch down to zero. Since we do not observe any PDE solution to the system \eqref{r-1} -- \eqref{r-4}  that leads to a touch-down singularity in finite time, we leave the discussion of the traveling wave solution that touches down to zero, as considered in Lemma \ref{lem-tw}, for future work.

\section{Numerical studies}
\label{sec:numerics}
In this section, we numerically investigate the coupled PDE system \eqref{r-1} -- \eqref{r-4} to explore the fibre coating dynamics and verify the analytical results in previous sections. Following the work of Ruan et al. \cite{ruan2021liquid}, we specify the form of the function $g(h)$ based on two models - the plug flow model and the laminar flow model. For the plug flow model, we set $g(h)$ based on the form in \eqref{plug_model}. For the laminar flow model, the function $g(h)$ takes the form in equation \eqref{laminar_model}.

Firstly, we numerically investigate the travelling wave solutions $(H(\xi), U(\xi))$ that satisfy the coupled ODE system \eqref{t-1} - \eqref{t-2} with the mass constraint \eqref{mass}, $\int_0^L V(\xi)~d\xi = \int_0^L (H^2(\xi)-1)~d\xi = M$.
We apply Newton's method to solve this nonlinear eigenvalue problem, where the speed $s$ is treated as an unknown variable. The coupled differential equations are discretized for $0 \le \xi\le L$ with periodic boundary conditions on $H$ and $U$ by second-order center finite differences. An additional constraint $H(\xi^*) = H^*$ for some $0\le \xi^* \le L$ is imposed to guarantee the local uniqueness of the solution.

Figure \ref{fig:tws} presents typical travelling wave solutions $(H(\xi), U(\xi))$ corresponding to two cases for the plug flow model and two cases for the laminar flow model:
\begin{enumerate}
    \item[(a)] Plug flow: $a=0.2, b=10, c=1$ with travelling speed $s = 1.396$
    \item[(b)] Plug flow: $a=0.4, b=12, c=3$ with travelling speed  $s=2.517$;
    \item[(c)] laminar flow $a=1.5, b=13, c=4$ with travelling speed  $s=1.482$;
    \item[(d)] laminar flow $a=0.1, b=11, c=4$ with travelling speed $s=0.1$.
\end{enumerate}
The choices of the parameters $a$, $b$, and $c$ for these cases correspond to the typical traveling waves simulated and discussed in \cite{ruan2021liquid}. These parameter values result in typical traveling wave solutions for both plug flow and laminar flow models of varying magnitudes.
A fixed domain size $L = 20$ and mass constraint $M = 84.8$ are set for all cases.
The profiles are shifted so that the maximum of the droplet peaks are located at $\xi = L/2$. This comparison shows that in a fixed domain with equal volumes, the travelling waves for the plug flow model have more prominent peaks and higher velocity  magnitude than those obtained from the laminar flow model.

\begin{figure}
    \centering
    \includegraphics[width=5cm]{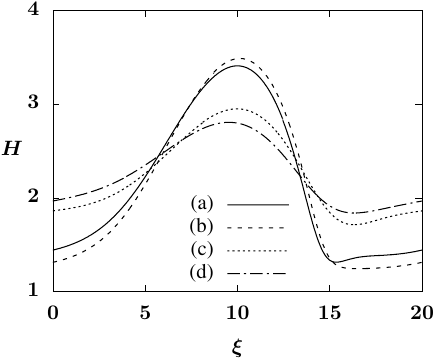}
    \includegraphics[width=5cm]{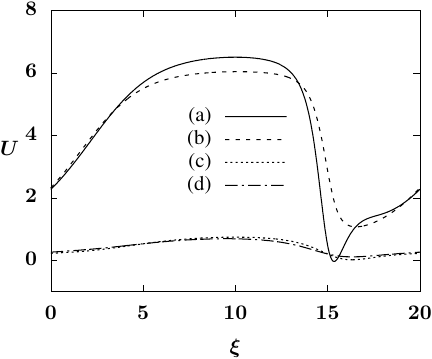}
    \caption{Typical travelling wave profiles (left) $H(\xi)$ and (right) $U(\xi)$ for two plug flow cases ((a) and (b)) and two laminar flow cases ((c) and (d)).
    }
    \label{fig:tws}
\end{figure}

Next, we study the transient PDE solutions of the governing model \eqref{r-1} -- \eqref{r-4} and verify the derived energy and entropy estimates in previous sections.
To numerically solve the coupled fourth-order PDEs, we use the Keller box method \cite{keller1971new} to decompose the model into a system of first-order differential equations,
\begin{equation}
\begin{split}
    & k = h_x,\quad p =k_x,\quad w = u_x, \\
    & u_t + a\left(\frac{u^2}{2}\right)_x+b\left[f(k)h^{-1}-f^3(k)p\right]_x = c\frac{[(h^2-1)w]_x}{h^2-1} + 1 - \frac{u}{g(h)},\\
    &   2hh_t + a[u(h^2-1)]_x = 0.
  \end{split}
\label{eqn:numeric_main}
\end{equation}
\begin{figure}
    \centering
    \includegraphics[width=5cm]{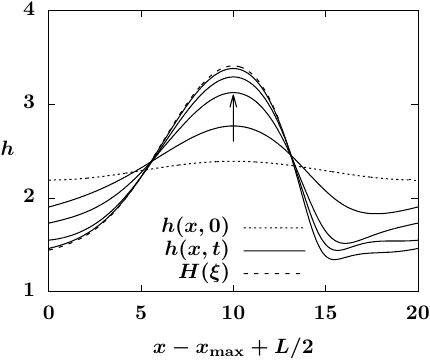}
    \includegraphics[width=5cm]{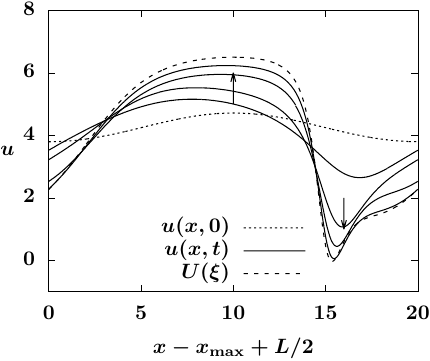}
    \includegraphics[width=5cm]{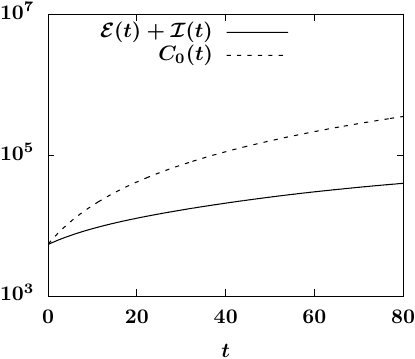}
    \includegraphics[width=5cm]{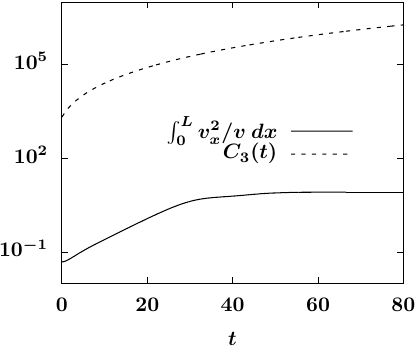}
    \caption{Dynamics of plug flow with (top left) $h(x,t)$ and (top right) $u(x,t)$ starting from initial profiles \eqref{eq:ic} with $h_0 = 2.29$, showing that the PDE solution approaches a travelling wave solution $(H(\xi), U(\xi))$ satisfying equations \eqref{t-1} -- \eqref{t-2} with the velocity $s=1.396$. The solutions are shifted so that the maximums are aligned. The corresponding energy (bottom left) satisfies the estimate \eqref{e-3}, $\mathcal{E}(t) + \mathcal{I}(t) < C_0(t)$, where $\mathcal{I}(t) = c \iint \limits_{Q_t} {  v u^2_x \,dx dt} + \iint \limits_{Q_t} { \tfrac{u^2 v }{g(h)} \,dx dt}$. The entropy (bottom right) satisfies the estimate \eqref{n-7}, $\int_0^L v_x^2/v~dx < C_3(t)$. The system parameters are $L = 20$, $a=0.2$, $b=10$, $c = 1$ with $g(h) = h^2-1$. }
    \label{fig:PDE_plugFlow_hu_convergence}
\end{figure}
Starting from the initial fluid film radius and the initial velocity
\begin{equation}
    h(x,0) = h_0 + 0.1 \sin(2\pi x / L),\qquad u(x,0) = g(h(x,0)),
\label{eq:ic}
\end{equation}
we solve the system \eqref{eqn:numeric_main} using fully implicit second-order centered finite differences over the domain $0 \leqslant x \leqslant L$, with periodic boundary conditions imposed on both $u$ and $h$. 
Here, we follow the work of Ruan et al. \cite{ruan2021liquid}, using the axial velocity profile $g(h(x,0))$ as the initial velocity profile.
For all PDE simulations, we keep the domain size $L = 20$ and $h_0 = 2.29$, so that the mass $M = 84.8$. The values chosen for $h_0$ and $L$ correspond to those used in a experimental comparison conducted in \cite{ruan2021liquid}. 

The top two plots in Figure \ref{fig:PDE_plugFlow_hu_convergence} show the dynamics of $(h(x,t), u(x,t))$ for the plug flow case, where the PDE solution converges to a travelling wave solution $(H(\xi), U(\xi))$ that satisfies the ODE system \eqref{t-1} -- \eqref{t-2} with the velocity $s=1.396$. The solution profiles are shifted by $x \to x - x_{\max}(t)+L/2$, where $x_{\max}(t)$ is the location of the peaks of the travelling waves in time, so that the peaks are aligned as the wave evolves and travels to the right.
The system parameters are given by $a=0.2$, $b=10$, $c = 1$ with $g(h) = h^2-1$ in \eqref{plug_model}, and the traveling wave corresponds to the case $(a)$ presented in Fig.~\ref{fig:tws}.

\begin{figure}
    \centering
    \includegraphics[width=5cm]{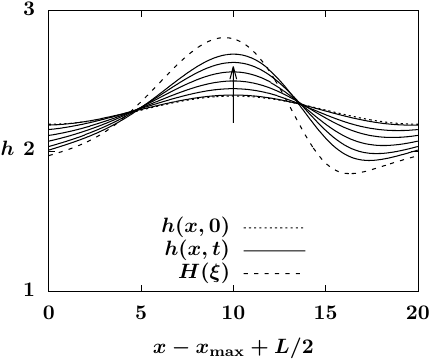}
    \includegraphics[width=5cm]{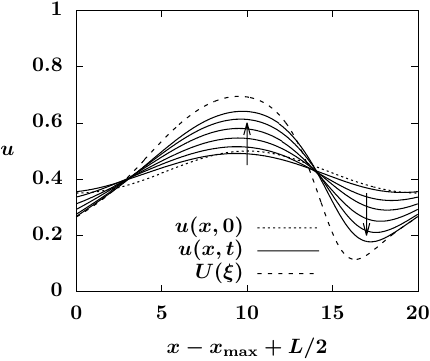}
    \includegraphics[width=5cm]{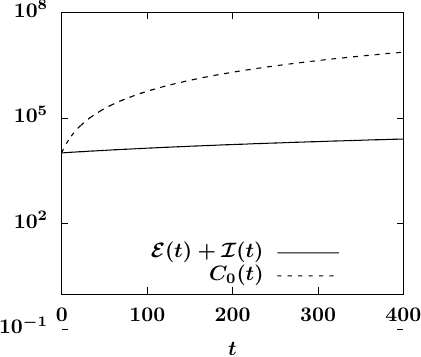}
    \includegraphics[width=5cm]{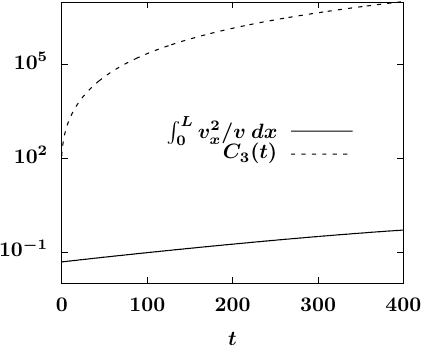}
    \caption{Dynamics of laminar flow (top left) $h(x,t)$ and (top right) $u(x,t)$ starting from initial profiles \eqref{eq:ic} with $h_0 = 2.29$, showing that the PDE solution approaches a travelling wave solution $(H(\xi), U(\xi))$ satisfying equations \eqref{t-1} - \eqref{t-2} with the velocity $s=0.1$. The solutions are shifted so that the maximums are aligned.
    Again, the corresponding energy plot (bottom left) shows that the energy satisfies the estimate \eqref{e-3},  $\mathcal{E}(t) + \mathcal{I}(t) < C_0(t)$, where $\mathcal{I}(t) = c \iint \limits_{Q_t} {  v u^2_x \,dx dt} + \iint \limits_{Q_t} { \tfrac{u^2 v }{g(h)} \,dx dt}$. The entropy (bottom right) satisfies the estimate \eqref{n-7}, $\int_0^L v_x^2/v~dx < C_3(t)$.  The system parameters are $L = 20$, $a=0.1$, $b=11$, $c = 4$, $g(h) = I(h)/(h^2-1)$. }
    \label{fig:PDE_laminar_hu_convergence}
\end{figure}
We also numerically verify the analytically derived energy and entropy estimates. In Figure \ref{fig:PDE_plugFlow_hu_convergence}~(bottom left), we show that the energy estimate \eqref{e-3}
$$\mathcal{E}(t) + \mathcal{I}(t) < C_0(t)$$ is satisfied as the transient PDE solution approaches the travelling wave profile in time,
where $\mathcal{I}(t) = c \iint \limits_{Q_t} {  v u^2_x \,dx dt} + \iint \limits_{Q_t} { \tfrac{u^2 v }{g(h)} \,dx dt}$.  Fig.~\ref{fig:PDE_plugFlow_hu_convergence} (bottom right) presents the numerically approximated integral $\int \limits_0^L { \tfrac{v_x^2}{v}\, dx }$ and the upper bound $C_3(t)$ defined in  \eqref{eq:C3} for the dynamic PDE solution, indicating that the entropy estimate \eqref{n-7},
$$
\int \limits_0^L { \tfrac{v_x^2}{v}\, dx } < C_3(t)
$$
is also satisfied in time. Here, we set $\epsilon=0.5$ in the definition of $C_3(t)$ in \eqref{eq:C3}, and this estimate holds for any $\epsilon\in (0,1)$.

Figure \ref{fig:PDE_laminar_hu_convergence} shows a similar numerical study for the laminar flow case with the system parameters $a = 0.1$, $b = 11$, and $c = 4$, and the function $g(h)$ is given by \eqref{laminar_model}. In this case, the laminar flow fluid radius $h(x,t)$ and velocity $u(x,t)$ starting from identical initial data used in Fig.~\ref{fig:PDE_plugFlow_hu_convergence} converge to a slowly-moving travelling wave parametrized by $(H(\xi), U(\xi))$ with the speed of propagation $s = 0.1$ (see Fig.~\ref{fig:PDE_laminar_hu_convergence} (top panel)). The obtained travelling wave corresponds to the case $(d)$ presented in Fig.~\ref{fig:tws}. Similar to the plug-flow case shown in Fig.~\ref{fig:PDE_plugFlow_hu_convergence}, the laminar flow solution also satisfies the energy and entropy estimates, as demonstrated in Fig.~\ref{fig:PDE_laminar_hu_convergence} (bottom panel).

\section{Conclusions}
\label{sec:conclusion}
The main contribution of this paper is the proof of the existence of weak solutions to the coupled PDE system \eqref{r-1}--\eqref{r-2} for the control-volume fibre coating model. This result establishes the analytical foundation for the control-volume model in real-world applications. The a priori energy-entropy functional estimates used in the proof also provide a possible pathway for showing the regularity of solutions in similar coupled PDE systems in other fibre coating models \cite{ruyer2008modelling, Ji2020Modeling}.
In contrast to the work of Bresch et al.~\cite{bresch2015two} and Kitavtsev et al.~\cite{kitavtsev2011weak}, for the proof of existence, we use another approximation of the continuity equation by the family of thin film equations (see \eqref{aprr-2}).  This new idea can be applied for the analysis of other systems with the same structure.
Typical numerical simulations of the PDE model are presented to support the analytical results, with a focus on the travelling wave solutions. For future studies, it would be interesting to further investigate the convergence criteria of PDE solutions to travelling wave solutions and other coherent structures.

\bibliographystyle{abbrv}
\bibliography{fiberSystem.bib}

\end{document}